\documentclass[10pt]{amsart}
\usepackage{amssymb}
\usepackage{dsfont}
\usepackage{amscd}
\usepackage{mathabx}
\usepackage[mathscr]{euscript} 
\usepackage[all]{xy}
\usepackage{url}
\usepackage{stmaryrd}
\usepackage{comment}
\usepackage[retainorgcmds]{IEEEtrantools}
\usepackage{hyperref}
\title{Model Theory of Fields with Virtually Free Group Actions}
\author[\"{O}. BEYARSLAN]{\"{O}zlem Beyarslan$^{\dagger}$}
\address{$^{\dagger}$Bo\v{g}azi\c{c}i \"{U}niversitesi}
\email{ozlem.beyarslan@boun.edu.tr}
\author[P. KOWALSKI]{Piotr Kowalski$^{\spadesuit}$}
\thanks{$^{\spadesuit}$ Supported by T\"{u}bitak grant 2221 and by the Narodowe Centrum Nauki grants no. 2015/19/B/ST1/01150 and 2015/19/B/ST1/01151.}
\address{$^{\spadesuit}$Instytut Matematyczny\\
Uniwersytet Wroc{\l}awski\\
Wroc{\l}aw\\
Poland}
\email{pkowa@math.uni.wroc.pl} \urladdr{http://www.math.uni.wroc.pl/\textasciitilde pkowa/ }

\DeclareMathOperator{\locus}{locus}
  
 \DeclareMathOperator{\aut}{Aut} \DeclareMathOperator{\id}{id}
 \DeclareMathOperator{\fr}{Fr} 
  
\DeclareMathOperator{\im}{im}  \DeclareMathOperator{\gal}{Gal}

 \DeclareMathOperator{\theo}{Th}\DeclareMathOperator{\alg}{alg}

\DeclareMathOperator{\coli}{\underrightarrow{\lim}}

\DeclareMathOperator{\spec}{Spec}\DeclareMathOperator{\rat}{rat}

\DeclareMathOperator{\acfa}{ACFA}\DeclareMathOperator{\tcf}{TCF}
\DeclareMathOperator{\ntp}{NTP}

\newtheorem{theorem}{Theorem}[section]
\newtheorem{prop}[theorem]{Proposition}
\newtheorem{lemma}[theorem]{Lemma}
\newtheorem{cor}[theorem]{Corollary}
\newtheorem{fact}[theorem]{Fact}
\newtheorem{conjecture}[theorem]{Conjecture}
\theoremstyle{definition}
\newtheorem{definition}[theorem]{Definition}
\newtheorem{example}[theorem]{Example}
\newtheorem{remark}[theorem]{Remark}
\newtheorem{question}[theorem]{Question}

\newtheorem{assumption}[theorem]{Assumption}

\begin{document}
\newcommand{\lili}{\underleftarrow{\lim }}
\newcommand{\coco}{\underrightarrow{\lim }}
\newcommand{\twoc}[3]{ {#1} \choose {{#2}|{#3}}}
\newcommand{\thrc}[4]{ {#1} \choose {{#2}|{#3}|{#4}}}
\newcommand{\Zz}{{\mathds{Z}}}
\newcommand{\Ff}{{\mathds{F}}}
\newcommand{\Cc}{{\mathds{C}}}
\newcommand{\Rr}{{\mathds{R}}}
\newcommand{\Nn}{{\mathds{N}}}
\newcommand{\Qq}{{\mathds{Q}}}
\newcommand{\Kk}{{\mathds{K}}}
\newcommand{\Pp}{{\mathds{P}}}
\newcommand{\ddd}{\mathrm{d}}
\newcommand{\Aa}{\mathds{A}}
\newcommand{\dlog}{\mathrm{ld}}
\newcommand{\ga}{\mathbb{G}_{\rm{a}}}
\newcommand{\gm}{\mathbb{G}_{\rm{m}}}
\newcommand{\gaf}{\widehat{\mathbb{G}}_{\rm{a}}}
\newcommand{\gmf}{\widehat{\mathbb{G}}_{\rm{m}}}
\newcommand{\ka}{{\bf k}}
\newcommand{\ot}{\otimes}
\newcommand{\si}{\mbox{$\sigma$}}
\newcommand{\ks}{\mbox{$({\bf k},\sigma)$}}
\newcommand{\kg}{\mbox{${\bf k}[G]$}}
\newcommand{\ksg}{\mbox{$({\bf k}[G],\sigma)$}}
\newcommand{\ksgs}{\mbox{${\bf k}[G,\sigma_G]$}}
\newcommand{\cks}{\mbox{$\mathrm{Mod}_{({A},\sigma_A)}$}}
\newcommand{\ckg}{\mbox{$\mathrm{Mod}_{{\bf k}[G]}$}}
\newcommand{\cksg}{\mbox{$\mathrm{Mod}_{({A}[G],\sigma_A)}$}}
\newcommand{\cksgs}{\mbox{$\mathrm{Mod}_{({A}[G],\sigma_G)}$}}
\newcommand{\crats}{\mbox{$\mathrm{Mod}^{\rat}_{(\mathbf{G},\sigma_{\mathbf{G}})}$}}
\newcommand{\crat}{\mbox{$\mathrm{Mod}^{\rat}_{\mathbf{G}}$}}
\newcommand{\cratinv}{\mbox{$\mathrm{Mod}^{\rat}_{\mathbb{G}}$}}
\newcommand{\ra}{\longrightarrow}
\newcommand{\gtcf}{\mbox{$G-\tcf$}}


\maketitle
\begin{abstract}
For a group $G$, we define the notion of a \emph{$G$-kernel} and show that the properties of $G$-kernels are closely related with the existence of a model companion of the theory of Galois actions of $G$. Using Bass-Serre theory, we show that this model companion exists for virtually free groups generalizing the existing results about free groups and finite groups. We show that the new theories we obtain are not simple and not even NTP$_2$.
\end{abstract}

\section{Introduction}
In this paper, we present a geometric axiomatization for the model companion of the theory of fields with actions of a fixed group $G$, where $G$ is of some specific type. Note that such an axiomatization is impossible in the case of $G=\Zz\times \Zz$,
see \cite{Kikyo1}. On the positive side, such an axiomatization is well-known in the case of $G=\Zz$ (the theory ACFA, see \cite{acfa1}), more generally in the case of $G=F_n$ (the theory $\acfa_n$, see \cite{Hr9}, \cite{KiPi}, \cite[Theorem 16]{Sjo} and \cite[Proposition 4.12]{MS2}) and finally for $G=\Qq$ (the theory $\Qq$ACFA, see \cite{med1}). It is also known to exist in the case of a finite group $G$, giving the theory $\gtcf$, see \cite{Sjo} and \cite{HK3}.

The aim of this paper is to put the results about free groups and about the finite groups in a natural common context. We consider the case of a finitely generated group $G$ having a free subgroup of finite index, i.e. a \emph{virtually free} group $G$. The class of virtually free finitely generated groups contains many interesting examples including the infinite dihedral group $D_{\infty}$ (the group which inspired our investigations) and, more generally, the groups of the form $G\ast H$, where $G$ and $H$ are finite.

To find our axioms, we analyze some kind of ``geometric prolongation process'' which allows to extend partial endomorphisms of fields to automorphisms of some bigger fields. This process is (slightly) visible e.g. in the proof of Theorem (1.1) from \cite{acfa1}, where one finds the following sentence:
\smallskip
\\
``By definition of $\sigma(U)$, $\sigma$ extends to an isomorphism from $K(a)$ onto $K(b)$ which
sends $a$ to $b$; this $\sigma$ in turn extends to an automorphism of $L$.''
\smallskip
\\
Our prolongation process may be seen as a constructive explanation of the part: ``this $\sigma$ in turn extends to an automorphism of $L$'' (the afore-mentioned isomorphism between $K(a)$ and $K(b)$ is called ``this $\sigma$'' above). We describe geometric conditions which allow our inductive extension process. Finding these conditions is rather easy in the case of ACFA (and also in the case of ACFA$_n$), but becomes more involved in the case where the acting group is not free.

We formalize the above process by introducing the notions of a \emph{transformal kernel} and its \emph{prolongation}. These notions were inspired by their differential counterparts from \cite{Lando} (and they were actually modeled on the difference case from \cite{cohn}). The transformal kernel above is actually the ring embedding $K(a)\to K(a,b)$, which maps $a$ to $b$ and extends $\sigma$ on $K$. The automorphisms of $L$ mentioned above is a prolongation of this transformal kernel. We formulate a general scheme of axioms which (in the case when a $G$-prolongation process exists) gives geometric axioms of the theory $\gtcf$, the model companion of the theory of fields with actions of $G$ by automorphisms (Theorem  \ref{algproofgen}).

In the case of a finitely generated virtually free group $G$, the Bass-Serre theory (see \cite{Dicks}) gives a convenient description of the group $G$. Namely, $G$ can be obtained from finite groups by the amalgamated free product construction (over a certain finite tree), followed by a finite sequence of HNN-extensions. For groups obtained in such a way, we show that a $G$-prolongation process exists, thus the theory $\gtcf$ exists as well (Theorem \ref{mainthm}).

We list below the main results of this paper.
\begin{itemize}
\item We provide a general framework of axiom schemes for the theories of universal domains of $G$-fields (Theorem \ref{algproofgen}).

\item Using the framework above, we show that for a finitely generated virtually free group $G$, the theory $\gtcf$ exists (Theorem \ref{mainthm}).
\end{itemize}

The theory ACFA (called $\Zz-$TCF in our terminology) turned out to be very useful for applications outside model theory, notably in diophantine geometry (see e.g. \cite{Hr9} and \cite{Sc1}) and in algebraic dynamics (see e.g. \cite{ChaHr1}, \cite{ChaHr2} and \cite{medsc}). It is also very important for applications to show that a particular difference field, an algebraically closed field with ``non-standard Frobenius'', is a model of ACFA (see \cite{HrFro}). The theory ACFA is \emph{simple}, which was crucial for the above applications, since simple theories enjoy a nice theory of geometric interactions between definable sets.
Although our new theories are not simple (Theorem \ref{nonntp}), there is evidence that these theories are NSOP$_1$. The class of NSOP$_1$ theories has been recently analyzed in \cite{KapRam} and it is shown there that these theories provide a natural generalization of simple theories and that they have many properties of simple theories, which were important for applications in the case of ACFA. Hence, knowing that our new theories are NSOP$_1$ may result in similar applications. It is possible that finding natural models of the theories $\gtcf$ may yield applications as well, and such models were already found in the case of finite groups (see \cite[Cor. 3.31]{HK3}).

Virtually free groups have a surprising number of characterizations coming from different branches of mathematics, see e.g. Introduction in \cite{arauja}. We list some of them below:
\begin{itemize}
\item fundamental groups of finite graphs of finite groups;

\item groups that are recognized by pushdown automata;

\item groups whose Cayley graphs have finite tree width.
\end{itemize}
It would be very interesting (and rather unexpected) to have one more characterization coming from model theory, and we conjecture (Conjecture \ref{conj1}) that this is the case.

The paper is organized as follows. In Section \ref{secgen}, we introduce the notion of a transformal kernel and show how the properties of such kernels are related to the existence of a model companion of the theory of transformal fields. In Section \ref{secvfg}, using the Bass-Serre theory, we show that $\gtcf$ exists for a finitely generated, virtually free group $G$. In Section \ref{secmt}, we discuss the model-theoretic properties of the theories $\gtcf$ we have obtained. In Section \ref{seclast}, we give a new example of a group $G$ for which the theory $\gtcf$ does not exist and formulate a conjecture about the existence of the theory $\gtcf$ for an arbitrary group $G$.

\section{Transformal kernels and model companions}\label{secgen}

In this section, we introduce the notion of a \emph{transformal kernel} (modeled on the notion of a \emph{differential kernel} from \cite{Lando}), and show a  close relation between the properties of transformal kernels and the existence of model companion. Similar relations exist in the characteristic $0$ differential case (see \cite{OSan}), the positive characteristic differential case (see \cite{K2}) and the Hasse-Schmidt differential case (see \cite{K3}).

\subsection{Functor}
Assume that $(K,\sigma)$ is a difference field and $V,W$ are varieties over $K$. By a \emph{variety}, we always mean a $K$-irreducible $K$-reduced algebraic subvariety of $\Aa_K^n$ for some $n>0$. Hence, a variety is basically the same as a prime ideal in the ring $K[X_1,\ldots,X_n]$ for some $n>0$. When we extend the map $\sigma$ to an automorphism of the ring $K[X_1,\ldots,X_n]$ (mapping each $X_i$ to $X_i$), then it also acts on the prime spectrum of $K[X_1,\ldots,X_n]$, so we get the varieties $^{\sigma}V,^{\sigma}W$. For each $K$-morphism $\varphi:V\to W$, we get a corresponding $K$-morphism $^{\sigma}\varphi:{}^{\sigma}V\to {}^{\sigma}W$. It is easy to check that we have obtained an endo-functor on the category of $K$-varieties. We also get the following ring isomorphism (but not a $K$-algebra homomorphism!)
$$\sigma_V:K[V]\to K[^{\sigma}V]$$
extending $\sigma$ on $K$.
\begin{remark}
We use the notation $^{\sigma}V$ rather than the (possibly more popular) notation $V^{\sigma}$, since we want to emphasize that $\sigma$ always acts \emph{on the left} (as the functions usually do). It will become important later, when we are going to consider at the same time several automorphisms of $K$.
\end{remark}
We also have a map (denoted by the same symbol)
$$\sigma_V:V(K)\to {}^{\sigma}V(K)$$
which can be understood in this (rather naive) context, as applying $\sigma$ coordinate-wise. This map is \emph{not} a morphism (it would be a morphism in the category of \emph{difference varieties}, but we are not going into this direction here). However, it is still a natural map, in the sense that for any morphism $\varphi:V\to W$, the following diagram is commutative
\begin{equation*}
 \xymatrix{  V(K)  \ar[rr]^{\varphi_K} \ar[d]^{\sigma_V} &  &  W(K) \ar[d]^{\sigma_W} \\
{}^{\sigma}V(K) \ar[rr]^{\left({}^{\sigma}\varphi\right)_K} &  &  {}^{\sigma}W(K).}
\end{equation*}
\begin{lemma}\label{l0}
On the level of coordinate rings, we have the following commutative diagram:
\begin{equation*}
 \xymatrix{  K[V]  \ar[d]^{\sigma_V} &  &  K[W]\ar[ll]_{\varphi^*}  \ar[d]^{\sigma_W} \\
K[{}^{\sigma}V]  &  &  K[{}^{\sigma}W]\ar[ll]_{\left({}^{\sigma}\varphi\right)^*}.}
\end{equation*}
\end{lemma}
The commutative diagram in the above lemma will be the source of commutativity of many diagrams which will be used later.

We will need one more easy result about the map $\sigma_V$.
\begin{lemma}\label{l1}
Suppose $K\subset \Omega$ is a field extension, and $a\in V(\Omega),b\in {}^{\sigma}V(\Omega)$ be generic points over $K$. Then after the natural $K$-algebra identifications:
$$K[a]\cong K[V],\ \ K[b]\cong K[{}^{\sigma}V],$$
we have $\sigma_V(a)=b$.
\end{lemma}
\begin{proof}
If we represent $K[V]$ as $K[\bar{X}]/I(V)$, then $I({}^{\sigma}V)=I(V)^{\sigma}$ and the map $\sigma_V$ is induced by $\sigma$ on $K$ and
$$\sigma_V(X_i+I(V))=X_i+I({}^{\sigma}V).$$
Hence, we get $\sigma_V(a)=b$ after the natural $K$-algebra identifications.
\end{proof}

\subsection{Transformal kernels}\label{secacfa}
Let us fix a difference field $(K,\sigma)$. We introduce now the main definition of this section.
\begin{definition}\label{kerdef}
\begin{enumerate}
\item A \emph{transformal kernel} or a \emph{$\Zz$-kernel} (with respect to the difference field $(K,\sigma)$) is a tower of fields $K\subseteq L\subseteq L'$ together with a field homomorphism $\sigma':L\to L'$ such that $\sigma'|_K=\sigma$ and $L'=L\sigma'(L)$.

\item We denote the transformal kernel as in item $(1)$ by $(L,L',\sigma')$.

\item A transformal kernel $(L',L'',\sigma'')$ is a \emph{prolongation} of a transformal kernel $(L,L',\sigma')$, if $\sigma''|_{L}=\sigma'$.
\end{enumerate}
\end{definition}
\begin{remark}
Clearly, a difference field extension $(K,\sigma)\subseteq (L,\sigma')$ gives the transformal kernel $(L,L,\sigma')$ which is ``the best one'', and it is its own prolongation.
\end{remark}
We will now investigate a close relation between transformal kernels and the theory ACFA. We note first a rather obvious result and hint on its well-known (non-constructive, though) proof. Actually, a constructive version of its proof, which we will provide later, will make clear the connection between prolongations of transformal kernels and the theory ACFA.
\begin{prop}\label{noncons}
Each transformal kernel $(L,L',\sigma')$ has a prolongation which is a difference field.
\end{prop}
\begin{proof}
The argument is standard, one uses transcendence bases and the corresponding basic fact about actions of $\Zz$ on (pure) \emph{sets}.
\end{proof}
We recall that all the varieties considered in this section are $K$-irreducible affine algebraic varieties over $K$.
\begin{definition}
We call a pair of varieties $(V,W)$ a \emph{$\Zz$-pair}, if $W\subseteq V\times {}^{\sigma}V$ and both the projections
$$W\to V,\ \ W\to {}^{\sigma}V$$
are dominant.
\end{definition}
\begin{lemma}\label{lemma1}
Any $\Zz$-pair $(V,W)$ gives a $\Zz$-kernel of the form $(K(V),K(W),\sigma')$. Moreover, there is a $K$-generic point $a\in V(K(V))$ such that $(a,\sigma'(a))$ is a $K$-generic point of $W$ (hence $\sigma'(a)$ is a $K$-generic point of ${}^{\sigma}V$).
\end{lemma}
\begin{proof}
The dominant projection maps $W\to V,W\to {}^{\sigma}V$ induce the following $K$-algebra homomorphisms:
$$\pi^W_V:K(V)\to K(W),\ \ \ \ \pi^W_{{}^{\sigma}V}:K({}^{\sigma}V)\to K(W).$$
Using the map $\pi^W_V$, we identify $K(V)$ with a $K$-subalgebra of $K(W)$, and we define:
$$\sigma':=\pi^W_{{}^{\sigma}V}\circ \sigma_V,$$
where $\sigma_V$ is now considered as a map from $K(V)$ to $K({}^{\sigma}V)$.

For the moreover part, let as assume that $V\subseteq \Aa^n$ and define:
$$a:=\left(X_1+I(W),\ldots,X_n+I(W)\right),\ \ \ a':=\left(X_{n+1}+I(W),\ldots,X_{2n}+I(W)\right).$$
Then $(a,a')\in W(K(W))$ is a $K$-generic point of $W$, hence (since both the projections are dominant) $a\in V(K(W))$ is a $K$-generic point of $V$, and $a'\in V(K(W))$ is a $K$-generic point of ${}^{\sigma}V$. Therefore, we can apply Lemma \ref{l1} and see that after the natural identifications $K[a]\cong K[V],K[a']\cong K[{}^{\sigma}V]$, we get that $a'=\sigma'(a)$.
\end{proof}
The result below gives a natural notion of a prolongation of a $\Zz$-pair. We will mostly use in the sequel the equivalence between the items $(1)$ and $(3)$.
\begin{lemma}\label{mainext}
Suppose that $(V,W)$ and $(W,W')$ are $\Zz$-pairs. Then the following are equivalent.
\begin{enumerate}
\item The $\Zz$-kernel $(K(W),K(W'),\sigma'')$ coming from the $\Zz$-pair $(W,W')$ (as in Lemma \ref{lemma1}) is a prolongation of the $\Zz$-kernel $(K(V),K(W),\sigma')$ coming from the $\Zz$-pair $(V,W)$ (as in Lemma \ref{lemma1}).
\item The following diagram is commutative (notation from the proof of Lemma \ref{lemma1}):
\begin{equation*}
 \xymatrix{  W' \ar[rr]^{\pi^{W'}_{W}} \ar[d]_{\pi^{W'}_{{}^{\sigma}W}} &  & W \ar[d]^{\pi^{W}_{{}^{\sigma}V}} \\
{}^{\sigma}W \ar[rr]^{\pi^{{}^{\sigma}W}_{{}^{\sigma}V}} &  &  ^{\sigma}V.}
\end{equation*}

\item $W'\subseteq W\times_{{}^{\sigma}V} {}^{\sigma}W$.
\end{enumerate}
\end{lemma}
\begin{proof}
Note that the item $(1)$ is equivalent to fact that the following equality between maps from $K(V)$ to $K(W')$ holds:
$$\pi^{W'}_{W}\circ \pi^{W}_{{}^{\sigma}V}\circ \sigma_V=\pi^{W'}_{{}^{\sigma}W} \circ \sigma_{W}\circ \pi^{W}_{V}.$$
For the proof of the equivalence between $(1)$ and $(2)$, it is enough to consider the following (big) diagram of $K$-algebra monomorphisms:
\begin{equation*}
 \xymatrix{  K(V) \ar[rr]^{\sigma_V}\ar[d]^{\pi_V^{W}} & & K({}^{\sigma}V) \ar[d]^{\pi_{{}^{\sigma}V}^{{}^{\sigma}W}} \ar[rr]^{\pi_{{}^{\sigma}V}^{W}} & &  K(W) \ar[d]^{\pi_{W}^{W'}} \\
 K(W) \ar[rr]^{\sigma_{W}} & & K({}^{\sigma}W) \ar[rr]^{\pi_{{}^{\sigma}W}^{W'}} & &  K(W'),}
\end{equation*}
and notice that the left-hand side (small) diagram commutes by Lemma \ref{l0}.
\\
The equivalence between $(2)$ and $(3)$ is immediate.
\end{proof}
We will start now from a $\Zz$-kernel and obtain a $\Zz$-pair. First, we need an easy result about loci which we leave without proof.
\begin{lemma}\label{loci}
Suppose $(L,L',\sigma')$ is a $\Zz$-kernel, $a\in L^n$ and $V=\locus_K(a)$. Then we have:
$${}^{\sigma}V=\locus_K(\sigma'(a)).$$
\end{lemma}
We show below a converse of Lemma \ref{lemma1}.
\begin{lemma}\label{lemma2}
Any finitely generated $\Zz$-kernel comes from a $\Zz$-pair in the way described in Lemma \ref{lemma1}.
\end{lemma}
\begin{proof}
Let $(L,L',\sigma')$ be a $\Zz$-kernel. Take a finite tuple $a$ in $L$ such that $L=K(a)$. We define
$$b:=(a,\sigma'(a)),\ \ V:=\locus_K(a),\ \ W:=\locus_K(b).$$
By Lemma \ref{loci}, we have ${}^{\sigma}V=\locus_K(\sigma'(a))$. Since $L'=K(b)$, $(V,W)$ is a $\Zz$-pair which produces (using Lemma \ref{lemma1}) the original $\Zz$-kernel $(L,L',\sigma')$.
\end{proof}
We note below a very general result which will be used for the category of varieties, where epimorphisms coincide with dominant maps.
\begin{fact}\label{genfact}
Let $\mathcal{C}$ be a category with fiber products and $B_1\to A,\ldots,B_n\to A$ be epimorphisms in $\mathcal{C}$. Then for all $i\in \{1,\ldots,n\}$, the corresponding projection morphism
$$B_1\times_AB_2\times_A\ldots\times_AB_n\to B_i$$
is an epimorphism as well.
\end{fact}
We show now a (preparatory) constructive version of Proposition \ref{noncons}.
\begin{lemma}[$\Zz$-prolongation lemma]\label{zprol}
Let $(V,W)$ be a $\Zz$-pair and define
$$W':=W\times_{{}^{\sigma}V} {}^{\sigma}W.$$
Then $(W,W')$ is a $\Zz$-pair and the corresponding $\Zz$-kernel $(K(W),K(W'),\sigma'')$ is a prolongation of the corresponding $\Zz$-kernel $(K(V),K(W),\sigma')$.
\end{lemma}
\begin{proof}
We need to check first that $(W,W')$ is a $\Zz$-pair. Since the morphism $W\to V$ is dominant, the morphism ${}^{\sigma}W\to {}^{\sigma}V$ is dominant as well, so $(W,W')$ is a $\Zz$-pair by Fact \ref{genfact}.

The prolongation statement concerning the corresponding $\Zz$-kernels follows from Lemma \ref{mainext}.
\end{proof}
\begin{remark}\label{acfa2}
For subsequent generalizations to the case of an arbitrary finitely generated marked group $(G,\rho)$, let us introduce the following notation:
$$\rho=(1,\sigma),\ \ \
\rho\rho=\left[\begin{array}{cc}
1          & \sigma         \\
\sigma     & \sigma^2
\end{array}\right].$$
We prefer to represent the sequence $(1,\sigma,\sigma,\sigma^2)$ in the matrix form, since this form makes the internal symmetries of the sequence $\rho\rho$ easier to visualize, which will become important when we will consider several automorphisms.

Let $(V,W)$ be a $\Zz$-pair, in particular $W\subseteq V\times {}^{\sigma}V$. We define:
$${}^{\rho}V:=V\times {}^{\sigma}V,\ \ \ {}^{\rho\rho}V:=V\times {}^{\sigma}V\times {}^{\sigma}V\times {}^{\sigma^2}V.$$
We also denote:
$${}^{\rho\cdot \rho}V:={}^{\rho\rho}V\cap \Delta^3_2,$$
i.e. we identify in ${}^{\rho\cdot \rho}V$ the second coordinate with the third coordinate (the only pair of coordinates in ${}^{\rho\rho}V$ which can be identified). Note that then we have the following:
$${}^{\rho}W\cap {}^{\rho\cdot \rho}V=(W\times {}^{\sigma}W)\cap \Delta^3_2=W\times_{{}^{\sigma}V}{}^{\sigma}W.$$
This observation will allow us to extend the $\Zz$-prolongation process from Lemma \ref{zprol} to the case of other groups (in place of $\Zz$). To summarize, for a $\Zz$-pair $(V,W)$, we define
$$W':={}^{\rho}W\cap {}^{\rho\cdot \rho}V,$$
and Lemma \ref{zprol} tells us that $(W,W')$ is a $\Zz$-pair again.
\end{remark}
The actual constructive version of Proposition \ref{noncons} is stated below.
\begin{prop}\label{acfacase}
Suppose that $(V,W)$ is a \emph{$\Zz$-pair}. Then the $\Zz$-kernel coming from $(V,W)$ has a prolongation which is a difference field extension.
\end{prop}
\begin{proof}
By Lemma \ref{zprol}, there is a sequence of $K$-varieties
$$V_0=V,\ V_1=W,\ V_2=W',\ V_3,\ \ldots $$
such that
$$K(V)\subseteq K(V_1)\subseteq K(V_2)\subseteq K(V_3)\subseteq \ldots$$
and there are ring homomorphisms
$$\sigma_m:K(V_m)\to K(V_{m+1})$$
such that we have
$$\sigma_0=\pi_{{}^{\sigma}V}^W\circ \sigma_V,\ \  \sigma_m\subseteq \sigma_{m+1}.$$
If we take $L:=\bigcup_m K(V_m)$ and $\sigma':=\bigcup_m\sigma_m$, then $(L,\sigma')$ is a difference field extension of $(K,\sigma)$, which is also a prolongation of the $\Zz$-kernel coming from the $\Zz$-pair $(V,W)$.
\end{proof}

\subsection{Transformal kernels and axioms of ACFA}\label{secafca}
We phrase now the well-known geometric axioms of the theory ACFA in terms of $\Zz$-pairs. This reformulation is obvious, however, it has the flexibility needed for our intended generalizations to the case of actions of more general groups.
\smallskip
\\
\textbf{Axioms for ACFA}
\\
The structure $(K,\sigma)$ is a difference field such that for each $\Zz$-pair $(V,W)$, there is $x\in V(K)$ such that $(x,\sigma(x))\in W(K)$.
\begin{remark}\label{acfno}
Usually these axioms include an extra assumption that $K$ is algebraically closed. This extra assumption is not necessary to state axioms for ACFA, and the existentially closed models for actions of most of the groups are \emph{not} algebraically closed. More precisely, using Theorem \ref{sjres}(2), one can conclude that an existentially closed $G$-field is algebraically closed if and only if the profinite completion of $G$ is a projective profinite group. We show below that the above axioms are first-order, even in the case where the ground field is not algebraically closed.
\begin{enumerate}
\item Using \cite[Lemma 3.1]{HK3} (the proof generalizes to the case of an arbitrary group $G$ in a straightforward way), we may assume that the basic field is perfect. Hence our varieties are also \emph{geometrically reduced}, by e.g.
    \cite[\href{http://stacks.math.columbia.edu/tag/030V}{Tag 030V}]{stacks-project}.

\item The notion of $K$-irreducibility (over an \emph{arbitrary} field $K$) is first-order definable using the general bounds from \cite{vddsch}. It is explained in detail e.g. in \cite[Remark 2.7]{HK3}.

\item We need to be more careful when proving that the notion of a dominant morphism $W\to V$ is first-order definable. We will do it in several steps below.
\begin{enumerate}
\item As usual, the right definition of ``dominant'' is the schematic one, i.e. the corresponding map $K[V]\to K[W]$ should be one-to-one (which is exactly what we need to extend ring homomorphisms to fields).

\item Since tensoring over $K$ is an exact functor, the map $K[V]\to K[W]$ is one-to-one if and only if the map
$$K[V]\otimes_KK^{\alg}\to K[W]\otimes_KK^{\alg}$$
is one-to-one.

\item For any algebraic variety $T$ over $K$, we denote
$$T':=T\times_{\spec(K)}\spec(K^{\alg})$$
(change of basis from $K$ to $K^{\alg}$). Then we have:
$$K^{\alg}[T']\cong_{K^{\alg}}K[T]\otimes_KK^{\alg}.$$
Hence we need to express in a definable fashion that the morphism $W'\to V'$ is dominant.

\item Since $V$ is $K$-irreducible, the absolute Galois group of $K$ acts transitively on the set of irreducible components of $V'$, hence $V'$ is equi-dimensional. Similarly, $W'$ is equi-dimensional.

\item Since $V'$ and $W'$ are equi-dimensional, any morphisms $f:W'\to V'$ is dominant if and only if
$$\dim\left(V'\setminus f(W')\right)<\dim(V').$$
The last condition is quantifier-free definable over $K$.

\item Since we get a definable over $K$ condition which is quantifier-free, the dominance condition holds in $K^{\alg}$ if and only if it holds in $K$, which finishes the argument.
\end{enumerate}
\end{enumerate}
\end{remark}
We briefly recall, and phrase in our terminology, a proof of the following result (see \cite{acfa1}) saying that the axioms above do axiomatize the class of existentially closed difference fields. In next sections, we will use a similar, but technically more complicated, procedure to axiomatize theories of difference fields for actions of other groups.
\begin{theorem}\label{thmacfa}
A difference field $(K,\sigma)$ is existentially closed if and only if, it is a model of the theory $\acfa$.
\end{theorem}
\begin{proof}
$(\Rightarrow)$ Since $(K,\sigma)$ is existentially closed, it is enough to find a difference field extension $(K,\sigma)\subseteq (L,\sigma')$ and $a\in V(L)$ such that $(a,\sigma'(a))\in W(L)$. By Lemma \ref{lemma1}, there is a $\Zz$-kernel corresponding to the $\Zz$-pair $(V,W)$. By Proposition \ref{acfacase}, the $\Zz$-kernel corresponding to the $\Zz$-pair $(V,W)$ has a prolongation $(L,\sigma')$ which is a difference field extension, and there is $a\in V(L)$ such that $(a,\sigma'(a))\in W(L)$.
\\
$(\Leftarrow)$ Suppose now that $(K,\sigma)\models \mathrm{ACFA}$. Let $\varphi(x)$ be a quantifier-free formula over $K$ in the language of difference fields. As usual, we can assume that:
$$\varphi(x):\ \ \ F_1(x,\sigma(x))=0\wedge \ldots\wedge F_m(x,\sigma(x))=0\ \wedge\ H(x,\sigma(x))\neq 0,$$
where $F_1,\ldots,F_m,H\in K[X,X']$ for some $m\in \Nn$ (the length of $X$ and $X'$ is the same as the length of the variable $x$). After replacing $x$ with $(x,y)$ and $H$ with $HY-1$, we can also assume that $\varphi(x)$ is of the form $\bigwedge F_i(x,\sigma(x))=0$.

Assume that there is a difference field extension $(K,\sigma)\subseteq (L,\sigma')$ such that
$$(L,\sigma')\models \exists x\ \varphi(x).$$
Let $a$ be a tuple in $L$ satisfying $\varphi(x)$. Then
$$\left(K(a),K(a,\sigma'(a)),\sigma'|_{K(a)}\right)$$
is a finitely generated $\Zz$-kernel. By Lemma \ref{lemma2}, there is a $\Zz$-tuple $(V,W)$ corresponding to this $\Zz$-kernel. Since $(K,\sigma)\models \mathrm{ACFA}$, there is $x\in V(K)$ such that $(x,\sigma(x))\in W(K)$ which exactly means that $(K,\sigma)\models \exists x\ \varphi(x)$.
\end{proof}
\begin{remark}\label{algproof}
It is clear that the main ingredients in the above proof were Proposition \ref{acfacase} and  Lemma \ref{lemma2}. It is also clear that the crucial Proposition \ref{acfacase} solely depends on the $\Zz$-prolongation Lemma, which is Lemma \ref{zprol}.

In Section \ref{wordsec}, we will describe the right conditions which are necessary to carry on the prolongation process in a general case (i.e. for a group which is not necessarily free), and we will show that a generalization of Theorem \ref{thmacfa} holds, if these conditions are satisfied (Theorem \ref{algproofgen}).
\end{remark}
We quickly see below that all the arguments of this section immediately generalize to the case of several automorphisms to give the (also well-known) theory ACFA$_n$. We fix now a difference field $(K,\sigma_1,\ldots,\sigma_n)$, which we sometimes call an \emph{$F_n$-field}.
\begin{enumerate}
\item Definition \ref{kerdef} has an obvious generalization here to give a notion of an \emph{$F_n$-kernel} (we demand now that $L'=L\sigma_1'(L)\ldots \sigma_n'(L)$) and the corresponding prolongation.

\item We call a pair of varieties $(V,W)$ an \emph{$F_n$-pair}, if $W\subseteq V\times {}^{\sigma_1}V\times \ldots \times {}^{\sigma_n}V$ and all the projections
$$W\to V,\ \ W\to {}^{\sigma_1}V,\ \ldots\ ,\ \ W\to {}^{\sigma_n}V$$
are dominant.

\item Lemma \ref{lemma1} has an obvious generalization to this case.

\item It is also easy to generalize the Lemma \ref{mainext} to the case of several automorphisms. Using the (multi-)diagonal notation from Remark \ref{acfa2}, the right generalization of condition $(3)$ from Lemma \ref{mainext} is
$$W'\subseteq \Delta^{n+2}_2\cap \Delta^{2n+3}_3\cap \ldots \cap \Delta^{n^2}_{n}\cap \Delta^{n^2+n+1}_{n+1},$$
or, following the notation from Remark \ref{acfa2}, in a more compact form as (for $\rho:=(1,\sigma_1,\ldots,\sigma_n)$):
$$W'\subseteq {}^{\rho}W\cap {}^{\rho\cdot \rho}V.$$

\item Using Fact \ref{genfact}, it is easy now to generalize the $\Zz$-prolongation Lemma (Lemma \ref{zprol}) to the case of the \emph{$F_n$-prolongation Lemma}, where we define $W'$ as
$$W':={}^{\rho}W\cap {}^{\rho\cdot \rho}V.$$

\item Lemma \ref{lemma2} and Proposition \ref{acfacase} also generalize to the case of several automorphisms in an obvious way.

\end{enumerate}
Using all the observations above, we can conclude as in the ACFA-case.
\smallskip
\\
\textbf{Axioms for ACFA$_n$}
\\
The structure $(K,\sigma_1,\ldots,\sigma_n)$ is a difference field such that for each $F_n$-pair $(V,W)$, there is $x\in V(K)$ such that $(x,\sigma_1(x),\ldots,\sigma_n(x))\in W(K)$.
\begin{theorem}
A difference field $(K,\sigma_1,\ldots,\sigma_n)$ is existentially closed if and only if, it is a model of the theory $\mathrm{ACFA}_n$.
\end{theorem}

\subsection{Word Problem Diagonals}\label{wordsec}
If we want to axiomatize the theory $\gtcf$ in the case when the group $G$ is not free, we need to find a way to encode in a first-order way the \emph{Word Problem} for a \emph{marked group} $(G,\rho)$, i.e. a group with a chosen sequence of generators. In this section, we expand Lemma \ref{mainext} to the case of relations between words (like commutativity), which may be satisfied by partial automorphisms.

Assume that $(G,\rho)$ is a marked group. Our sequence of generators is finite $\rho=(\rho_1,\ldots,\rho_m)$ and we always assume that $\rho_1=1$. A \emph{$G$-field} is a field $K$ together with an action of $G$ by automorphisms. We denote the corresponding field automorphisms of $K$ by the same symbols $\rho_1,\ldots,\rho_m$ and we consider a $G$-field as a first order structure in the following way: $(K;+,\cdot,\rho_1,\ldots,\rho_m)$.

We also assume that the marked group $(G,\rho)$ is finitely presented in a rather simple way, that is we make the following.
\begin{assumption}\label{assume}
We assume that for
$$P:=\left\{(i,j,k,l)\in \{1,\ldots,m\}^4\ |\ \rho_i\rho_j=\rho_k\rho_l\text{ (in $G$)}\right\},$$
$G$ has a presentation of the following form:
$$G=\left\langle\rho\ |\ \rho_i\rho_j=\rho_k\rho_l\text{ for }(i,j,k,l)\in P\right\rangle.$$
Thus, the set $P$ encodes the Word Problem for $G$.
\end{assumption}
We can define now the notion of a $G$-kernel, which can be understood as a ``two-step version'' of the notion of a $\Zz$-kernel.
\begin{definition}\label{kerdefg}
\begin{enumerate}
\item A \emph{$G$-kernel} (with respect to the $G$-field $(K,\rho)$) is a tower of fields $K\subseteq L\subseteq L'\subseteq L''$ together with field homomorphisms
    $$\rho_i':L\to L',\ \ \ \rho_i'':L'\to L''$$
    for $i=1,\ldots,m$ such that:
\begin{itemize}
\item each $\rho_i'$ extends $\rho_i$,

\item the homomorphisms $\rho_1'$ and $\rho_1''$ are inclusions,

\item for each $(i,j,k,l)\in P$, we have
$$\rho''_i\circ \rho'_j=\rho''_k\circ \rho'_l$$
(note that this condition implies that each $\rho''_i$ extends $\rho'_i$),

\item we have:
$$L'=\rho_1'(L)\ldots\rho_m'(L),\ \ \ L''=\rho_1''(L')\ldots\rho_m''(L').$$
\end{itemize}

\item We denote the $G$-kernel as in the item $(1)$ by $(L,L',L'',\rho',\rho'')$.

\item There is an obvious notion of a \emph{prolongation} of $G$-kernels as in Definition \ref{kerdef}(3), i.e. a $G$-kernel $(L_*,L'_*,L''_*,\rho'_*,\rho''_*)$ is a prolongation of a $G$-kernel $(L,L',L'',\rho',\rho'')$ if and only if, we have $L_*=L'$, $L_*'=L''$ and $\rho'_*=\rho''$.
\end{enumerate}
\end{definition}
\begin{remark}
As before, a $G$-field extension $(K,\rho)\subseteq (L,\rho')$ gives the $G$-kernel $(L,L,L,\rho',\rho')$ which is ``the best one'', and it is its own prolongation. The difference here is that there is no guarantee that a $G$-kernel has a prolongation which is a $G$-field extension. This seems to be the main reason for the (non-)existence of a model companion of the theory of $G$-fields for some groups $G$.
\end{remark}

Assume now for a moment that we just take one relation from $P$, i.e. assume that $(K,\sigma,\tau,\varepsilon,\delta)$ is a difference field such that $\tau\sigma=\varepsilon\delta$. We want to encode the word equality $\tau\sigma=\varepsilon\delta$ in a first-order way.

Let us set $\rho:=(1,\sigma,\tau,\varepsilon,\delta)$ and we assume $V,W,W'$ are $K$-irreducible varieties such that
$$W\subseteq V\times {}^{\sigma}V\times {}^{\tau}V\times {}^{\varepsilon}V\times {}^{\delta}V={}^{\rho}V,$$
$$W'\subseteq W\times {}^{\sigma}W\times {}^{\tau}W\times {}^{\varepsilon}W\times {}^{\delta}W={}^{\rho}W,$$
and all the projections are dominant. We have the appropriate field homomorphisms
$$\sigma',\tau',\varepsilon',\delta':K(V)\to K(W),\ \ \ \sigma'',\tau'',\varepsilon'',\delta'':K(W)\to K(W')$$
defined as in Lemma \ref{lemma1}.

Since ${}^{\varepsilon\delta}V={}^{\tau\sigma}V$, we have the corresponding ``non-trivial diagonal'':
$$\Delta^{\tau\sigma}_{\varepsilon\delta}\subseteq {}^{\rho\rho}V.$$
We need a version of Lemma \ref{mainext} which deals with the word problem for the group $G$. The following result is a ``Word Problem counterpart'' of Lemma \ref{mainext}.
\begin{prop}[Lemma on Word Problem]\label{lowp}
The following conditions are equivalent:
\begin{enumerate}
\item $\tau''\circ \sigma'=\varepsilon''\circ \delta'$,

\item $W'\subseteq \Delta^{\tau\sigma}_{\varepsilon\delta}$.
\end{enumerate}
\end{prop}
\begin{proof}
Consider the following commutative diagram:
\begin{equation*}
 \xymatrix{  K(V) \ar[rrd]_{\sigma_V} \ar[rrdd]_{\sigma'} \ar[rrrrd]^{(\tau\sigma)_V} & & & &  \\
& &  K({}^{\sigma}V)\ar[rr]_{\tau_{{}^{\sigma}V}} \ar[d]^{\pi_{{}^{\sigma}V}^{W}} & & K({}^{\tau\sigma}V)\ar[d]^{\pi_{{}^{\tau\sigma}V}^{{}^\tau W}}\\
 & & K(W)  \ar[rr]^{\tau_{W}} \ar[rrd]_{\tau''} &  &  K({}^{\tau}W) \ar[d]^{\pi_{{}^{\tau}W}^{W'}}\\
 & & & & K(W'),  }
\end{equation*}
as well as the corresponding diagram for $(\varepsilon,\delta)$ playing the role of $(\tau,\sigma)$. Since $(\tau\sigma)_V=(\varepsilon\delta)_V$ and both these maps are bijections, we get that the equality
$$\tau''\circ \sigma'=\varepsilon''\circ \delta'$$
is equivalent to the commutativity of the following diagram:
\begin{equation*}
 \xymatrix{  W' \ar[rr]^{\pi^{W'}_{{}^{\tau}W}} \ar[d]_{\pi^{W'}_{{}^{\varepsilon}W}} &  & {}^{\tau}W \ar[d]^{\pi^{{}^{\tau}W}_{{}^{\tau\sigma}V}} \\
{}^{\varepsilon}W \ar[rr]^{\pi^{{}^{\varepsilon}W}_{{}^{\varepsilon\delta}V}} &  & {}^{\varepsilon\delta}V={}^{\tau\sigma}V.}
\end{equation*}
The commutativity of the last diagram is equivalent to the condition $W'\subseteq \Delta^{\tau\sigma}_{\varepsilon\delta}$, similarly, as in Lemma \ref{mainext}.
\end{proof}
The above result gives a relatively easy criterion to check whether a prolongation of $(V,W)$ to  $(W,W')$ gives a right procedure potentially yielding a $G$-field. We prove below the counterpart of Lemma \ref{lemma1} in this context.
First, we introduce a notation generalizing the one from Remark \ref{acfa2}:
$${}^{\rho}V:={}^{\rho_1}V\times \dots \times {}^{\rho_m}V,$$
$${}^{\rho\rho}V:={}^{\rho_1}({}^{\rho}V)\times \dots \times {}^{\rho_m}({}^{\rho}V),$$
$${}^{\rho\cdot \rho}V:={}^{\rho\rho}V\cap \{\text{all possible diagonals}\}.$$
The last intersection means that to obtain ${}^{\rho\cdot \rho}V$ from ${}^{\rho\rho}V$, we identify ${}^{\rho_i}({}^{\rho_j}V)$ with ${}^{\rho_k}({}^{\rho_l}V)$ for all $(i,j,k,l)\in P$.
\begin{cor}\label{pe2}
Suppose that we have
$$W\subseteq {}^{\rho}V,\ \ W'\subseteq {}^{\rho}W\cap {}^{\rho\cdot\rho}V$$
and all the projections $W\to {}^{\rho_i}V,W'\to {}^{\rho_i}W$ are dominant. Then there is a $G$-kernel of the form
$$(K(V),K(W),K(W'),\rho',\rho'').$$
Moreover, there is a $K$-generic point $a\in V(K(V))$ such that $\rho'(a)$ is a $K$-generic point of $W$.
 \end{cor}
\begin{proof} The maps $\rho'_i,\rho''_i$ are obtained in the same was as in the proof of Lemma \ref{lemma1}. Proposition \ref{lowp} guarantees that for each $(i,j,k,l)\in P$, we have
$$\rho''_i\circ \rho'_j=\rho''_k\circ \rho'_l.$$
The moreover part follows again as in the proof of Lemma \ref{lemma1}.
\end{proof}
If we try now to generalize the observations made in Remark \ref{algproof}, we see that we need to find a good notion of a $G$-pair for which the ``$G$-prolongation Lemma'' (an analogue of  Lemma \ref{zprol}) would hold as well as an analogue of Lemma \ref{lemma2}. We comment here more on the general shape of the possible definition of a $G$-pair. A \emph{$G$-pair} should be a pair of varieties $(V,W)$ (over a field $K$ with a $G$-field structure) such that:
\begin{enumerate}
\item $W\subseteq {}^{\rho}V$,

\item all the projections $W\to {}^{\rho_i}V$ are dominant,

\item the pair $(V,W)$ satisfies a right ``$G$-iterativity condition''.
\end{enumerate}
Note that the conditions $(1)$ and $(2)$ are exactly the same as in the case of $F_n$-pairs (see the item $(2)$ below Remark \ref{algproof}). The most difficult task is to find the proper ``$G$-iterativity condition'' (taking care of the Word Problem for $(G,\rho)$), which appears in the condition $(3)$, and  to show that if $(V,W)$ is a $G$-pair, then $(W,W')$ is a $G$-pair as well. We formalize these observations below.
\begin{theorem}\label{algproofgen}
Suppose we have the notion of a $G$-pair as above and assume that we can show the following.
\begin{enumerate}
\item {\bf $G$-Prolongation Lemma}
\\
If $(V,W)$ is a $G$-pair, then $(W,{}^{\rho}W\cap {}^{\rho\cdot\rho}V)$ is a $G$-pair.

\item {\bf Analogue of Lemma \ref{lemma2}}
\\
Suppose that $$\left(K(a),K(\rho'(a)),K(\rho''(\rho'(a)))\right)$$ is a $G$-kernel. Then
$$\left(\locus_K(a),\locus_K(\rho'(a))\right)$$
is a $G$-pair.
\end{enumerate}
Then, the model companion of the theory of $G$-fields (the theory $\gtcf$) exists and we have the following.
\smallskip
\\
{\bf Axioms for $\gtcf$}
\\
For any $G$-pair $(V,W)$, there is $x\in V(K)$ such that $\rho(x)\in W(K)$.
\end{theorem}
\begin{proof}
This proof basically repeats the proof of Theorem \ref{thmacfa} in the more general context of $G$-fields.
\\
$(\Rightarrow)$ Since $(K,\rho)$ is existentially closed, it is enough to find a $G$-extension $(K,\rho)\subseteq (L,\widetilde{\rho})$ and $a\in V(L)$ such that $\widetilde{\rho}(a)\in W(L)$. Let $W':={}^{\rho}W\cap {}^{\rho\cdot\rho}V$. By the $G$-Prolongation Lemma, $(W,W')$ is a $G$-pair. In particular, the assumptions of Corollary \ref{pe2} are satisfied. By Corollary \ref{pe2},
$$(K(V),K(W),K(W'),\rho',\rho'')$$
is a $G$-kernel, and there is a $K$-generic point $a\in V(K(V))$ such that for each $\rho'(a)$ is a $K$-generic point of $W$. By repeated usage of the $G$-Prolongation Lemma (similarly as in the proof of Prop. \ref{acfacase}), this $G$-kernel has a prolongation $(L,\widetilde{\rho})$ which is a $G$-extension. Hence $a\in V(L)$ and $\widetilde{\rho}(a)\in W(L)$.
\\
$(\Leftarrow)$ Suppose now that $(K,\rho)\models G-\tcf$. Let $\varphi(x)$ be a quantifier-free formula over $K$. As in the proof of Theorem \ref{thmacfa}, we can assume that:
$$\varphi(x):\ \ \ F_1(x,\sigma(x))=0\wedge \ldots\wedge F_m(x,\sigma(x))=0,$$
where $F_1,\ldots,F_m\in K[X,X']$ for some $m\in \Nn$.
\\
Assume that there is a $G$-extension $(K,\rho)\subseteq (L,\rho')$ such that
$$(L,\rho')\models \exists x\ \varphi(x).$$
Let $a$ be a tuple in $L$ satisfying $\varphi(x)$. Then
$$\left(K(a),K(a,\rho'(a)),\rho'(\rho'(a)))\right)$$
is a finitely generated $G$-kernel. Let
$$V:=\left(\locus_K(a)\right),\ \ \ W:=\locus_K\left(\rho'(a)\right).$$
By Analogue of Lemma \ref{lemma2}, $(V,W)$ is a $G$-pair. By Axioms for $G-\tcf$, there is $x\in V(K)$ such that $\rho(x)\in W(K)$ which means that $(K,\rho)\models \exists x\ \varphi(x)$.
\end{proof}
\begin{remark}
Since the theory $(\Zz\times \Zz)-\tcf$ does not exists, then obviously the above procedure does not work for the case of $G=\Zz\times \Zz$. A practical obstacle is the following: for any possible notion of a $(\Zz\times \Zz)$-pair, if $(V,W)$ is a $(\Zz\times \Zz)$-pair and $W'$ is defined as in the proof of Theorem \ref{algproofgen}, then some projection $W'\to {}^{\rho_i}W$ will \emph{not} be dominant.
\end{remark}
In the next section, we will find a good notion of a $G$-pair and prove the corresponding $G$-prolongation lemma as well as an analogue of Lemma \ref{lemma2} for finitely generated, virtually free groups.

\section{Virtually free groups}\label{secvfg}

In this section, we show that the theory $\gtcf$ exists for a finitely generated, virtually free group $G$. To deal with the finite group case, we reformulate the results proven in \cite{HK3}. Then we use the Bass-Serre theory to pass from the case of finite groups to the case of virtually free groups. Since some virtually free groups have an explicit description as semi-direct products of a free group and a finite group, we also give explicit axioms in such cases (Section \ref{secgen1}), where we do not use the Bass-Serre theory.

\subsection{Finite groups}\label{n0}
We analyze here, using the terminology of this paper, the results from \cite{HK3}. Let $G_0=\{g_1,\ldots,g_e\}$ be a finite group,
$$\rho:=g:=(g_1,\ldots,g_e),$$
and $(K,g)$ be a $G_0$-field. We fix a pair of varieties $(V,W)$ such that $W\subseteq {}^{g}V$ and $W$ projects generically on each ${}^{g_i}V$. For each $i=1,\ldots,e$, we get the permutation $\lambda^i:G_0\to G_0$ induced by the left multiplication by $g_i$ and the corresponding algebraic morphisms:
$$\lambda^i_V: {}^{g}V\to {}^{g_ig}V;\ \ \ \lambda^V: {}^{g}V\to {}^{gg}V,\ \ \lambda^V:=(\lambda^1_V,\ldots,\lambda^e_V).$$
We also define (as in Theorem \ref{algproofgen}):
$$W':={}^{g}W\cap {}^{g\cdot g}V.$$
\begin{lemma}\label{finiteit}
The following are equivalent.
\begin{enumerate}
\item The pair $(V,W)$ satisfies the ``$G$-iterativity condition'' as in $(\clubsuit^{\mathrm{g}})$ from \cite[Remark 2.7(2)]{HK3}, which says that for each $i\leqslant e$, we have
$$\lambda^i_V(W)={}^{g_i}W.$$

\item $\lambda^V(W)\subseteq {}^{g}W$.

\item $\lambda^V(W)=W'$.

\item $W'$ projects generically on each ${}^{g_i}W$.
\end{enumerate}
Moreover, if the above equivalent conditions for $(V,W)$ are satisfied, then they are satisfied for $(W,W')$ as well.
\end{lemma}
\begin{proof}
We observe first that since ${}^{g\cdot g}V=\lambda^V({}^{g}V)$, we obtain the following:
\begin{equation}
{}^{g}W\cap {}^{g\cdot g}V=\lambda^V(W)\cap {}^{g\cdot g}V.\tag{$*$}
\end{equation}
The equivalence between $(1)$ and $(2)$ is clear from the definitions.

For the equivalence between $(2)$ and $(3)$, it is enough to use the equality $(*)$ above.

To see that $(3)$ implies $(4)$ and that $(4)$ implies $(1)$, it is enough to notice that for each $i=1,\ldots,e$ we have
$$\im(W'\to {}^{g_i}W)=\lambda^i_V(W)\cap {}^{g_i}W.$$

For the moreover part, one needs to notice that if the $(V,W)$ satisfies the equivalent conditions $(1)$--$(4)$, then
$$\lambda_W^i\left(\lambda^V(W)\right)={}^{g_i}\left(\lambda^V(W)\right),$$
for each $i\leqslant e$.
\end{proof}
\begin{remark}
\begin{enumerate}
\item The equivalent conditions from Lemma \ref{finiteit} give us the missing $G_0$-iterativity condition (which was discussed before Theorem \ref{algproofgen}). Hence, we have now the definition of a $G_0$-pair.

\item The moreover part of Lemma \ref{finiteit} is exactly the $G_0$-prolongation lemma.

\item As explained in \cite[Remark 2.7(3)]{HK3}, to obtain an axiomatization of the theory $G_0-\tcf$, one needs to consider only the varieties of the form $V=\Aa^n$.
\end{enumerate}
\end{remark}
Now, we only need a counterpart of Lemma \ref{lemma2} (below).
\begin{lemma}\label{sstrong0}
Suppose that $(L,L',L'',g',g'')$ is a $G_0$-kernel. For $a\in L^m$, let
$$V:=\locus_K(a),\ \ \ W:=\locus_K(g'(a)).$$
Then $(V,W)$ is a $G_0$-pair.
\end{lemma}
\begin{proof}
Clearly, we have $W\subseteq {}^{\rho}V$. Similarly as in the proof of Lemma \ref{lemma2}, one can show that all the projections $W\to {}^{g_i}V$ are dominant.  Therefore, we need to check only the $G_0$-iterativity condition from the definition of a $G_0$-pair. For each $i=1,\ldots,n$; we have (using an obvious analogue of Lemma \ref{loci}):
\begin{IEEEeqnarray*}{rCl}
{}^{g_i}W & = & \locus_K\left(g_i''\left(g_1'(a)\right),\ldots,g_i''\left(g_e'(a)\right)\right)\\
 & = & \locus_K\left(\lambda^i_V\left(g_1'(a),\ldots,g_e'(a)\right)\right) \\
 &= & \lambda^i_V(W).
\end{IEEEeqnarray*}
Hence $(V,W)$ is a $G_0$-pair.
\end{proof}
Using Theorem \ref{algproofgen}, we see that the below axioms describe the theory $G_0-\tcf$ (a model companion of the theory of fields with $G_0$-actions).
\smallskip
\\
\textbf{Axioms for $G_0-\tcf$}
\\
The structure $(K,g_1,\ldots,g_e)$ is a $G_0$-field such that for each $G_0$-pair $(V,W)$, there is $x\in V(K)$ such that $(g_1(x),\ldots,g_e(x))\in W(K)$.

\subsection{Amalgamated products}\label{secampr}
Let us assume that $G=B\ast_A C$, $B=\{b_1,\ldots,b_m\}$, $C=\{c_1,\ldots,c_l\}$ and $A=\{b_1,\ldots,b_k\}$, where $b_1=c_1,\ldots,b_k=c_k$. We define our sequence of generators of $G$ in the obvious way:
$$\rho:=(b_1,\ldots,b_m,c_{k+1},\ldots,c_l).$$
Let us fix a $G$-field $(K,\rho)$ and varieties $V,W$ such that $W\subseteq {}^{\rho}V$. We denote by $W_B$ be the (Zariski closure of the) projection of $W$ on ${}^{(b_1,\ldots,b_m)}V$, and by $W_C$ the (Zariski closure of the) projection of $W$ on ${}^{(c_1,\ldots,c_l)}V$.
\begin{definition}\label{def1am}
We say that $(V,W)$ is a \emph{$(B\ast_A C)$-pair}, if:
\begin{itemize}
\item $(V,W_B)$ is a $B$-pair;

\item $(V,W_C)$ is a $C$-pair.
\end{itemize}
\begin{remark}
If $(V,W)$ is a $(B\ast_A C)$-pair, then all the projections
$W\to {}^{\rho_i}V$ are automatically dominant.
\end{remark}
\end{definition}
\begin{example}\label{simpleamal}
It may be convenient to have in mind the simplest case where:
$$G=C_2\ast C_2=\langle\sigma\rangle\ast\langle\tau\rangle,\ \ \ \rho:=(1,\sigma,\tau),\ \ \ \rho\rho=\left[\begin{array}{ccc}
1          & \sigma     & \tau        \\
\sigma     & 1          & \sigma\tau  \\
\tau       & \tau\sigma &  1
\end{array}\right].$$
For $W\subseteq V\times {}^{\sigma}V\times {}^{\tau}V$, let $W_{\sigma}$ be the projection of $W$ on $V\times {}^{\sigma}V$, and $W_{\tau}$ be the projection of $W$ on $V\times {}^{\tau}V$. Then $(V,W)$ is a \emph{$(C_2\ast C_2)$-pair}, if $(V,W_{\sigma})$ is a $C_2$-pair and $(V,W_{\tau})$ is a $C_2$-pair.
\end{example}
We define as usual:
$$W':={}^{\rho}W\cap {}^{\rho\cdot \rho}V.$$
\begin{prop}[$(B\ast_A C)$-Prolongation Lemma]\label{g1astg2}
If $(V,W)$ is a $(B\ast_A C)$-pair, then $(W,W')$ is a $(B\ast_A C)$-pair as well.
\end{prop}
\begin{proof}
For any $b\in B$, let
$$\lambda^b_V:W_B\to {}^{b}W_B$$
be the coordinate permutation map, which is defined as in the beginning of Section \ref{n0}. We have the following:
$$W_B'={}^{b_1}W\times_{W_B}{}^{b_2}W\times_{W_B}\dots \times_{W_B}{}^{b_m}W,$$
where for each $b\in B$, the dominant morphism $\alpha_b:{}^{b}W\to W_B$ is defined as the following composition:
\begin{equation*}
 \xymatrix{{}^{b}W \ar[rr]^{{}^{b}\left(\pi^{W}_{W_B}\right)} &  &   {}^{b}W_B \ar[rr]^{\left(\lambda^b_V\right)^{-1}} &  & W_B.}
\end{equation*}
For any $b\in B$, we also have
$${}^{b}W_B'={}^{bb_1}W\times_{\lambda^b_V(W_B)}{}^{bb_2}W\times_{\lambda^b_V(W_B)}\dots \times_{\lambda^b_V(W_B)}{}^{bb_m}W,$$
where each dominant morphism ${}^{bb_i}W\to \lambda^b_V(W_B)$ is defined as ${}^{b}\alpha_{b_i}$.

By Fact \ref{genfact}, the dominance condition from the definition of a $B$-pair for $(W,W'_B)$ holds. For any $b\in B$, we also obtain
$$\lambda^b_W(W_B')={}^{b}W_{B}',$$
so the iterativity condition from the definition of a $B$-pair (appearing in Lemma \ref{finiteit}(1)) for $(W,W'_B)$ holds as well. Similarly, $(W,W'_C)$ is a $C$-pair, hence $(W,W')$ is a $(B\ast_A C)$-pair.
\end{proof}
\begin{remark}
Since the definition of a $(B\ast_A C)$-pair (Definition \ref{def1am}) was rather easy and general, one could wonder whether it is possible to:
\begin{enumerate}
  \item take groups $B,C$ such that we have ``good'' notions of $B$-pairs and $C$-pairs (i.e. both the $B$-Prolongation Lemma and the $C$-Prolongation Lemma hold);
  \item define the notion of a $(B\ast C)$-pair exactly as in Definition \ref{def1am};
  \item prove the $(B\ast C)$-Prolongation Lemma.
\end{enumerate}
One can definitely succeed with the items $(1)$ and $(2)$ above, e.g. by taking as $B$ and $C$ any finitely generated virtually free groups (\emph{after} proving the main result of this paper, i.e. Theorem \ref{mainthm}). However, there would be the following issues then.
\begin{itemize}
  \item The class of finitely generated virtually free groups is closed under taking free products, so (even after succeeding with the item $(3)$ above) we would not get anything new, i.e. we would not get any improvement of Theorem \ref{mainthm}. And because of Conjecture \ref{conj1}, we also do not see any way for a possible disproving of the item $(3)$. Actually, Conjecture \ref{conj1} (a posteriori) \emph{implies} that the item $(3)$ holds.

  \item Showing the item $(3)$ may be difficult in general (i.e. for an arbitrary ``good'' notions of $B$-pairs and $C$-pairs), since in the proof of Prop. \ref{g1astg2}, we used some particular properties of the notion of a $G$-pair for a finite group $G$. More specifically, the notion of a $G$-pair with respect to a finite $G$ ``behaves nicely'' with respect to fiber products, which was used in the proof of Prop. \ref{g1astg2}. One could probably describe axiomatically all the ``nice behaviour'' needed, but we do not think it would be very fruitful, because of the previous item above.
\end{itemize}
\end{remark}
Finally, we consider the most general case we need here, that is we assume that we have a finite \emph{tree of} finite \emph{groups} $(B(-),T)$. It means the following:
\begin{itemize}
\item $T$ is a tree with vertices $1,\ldots,t$ and edges of the form $(i,j)$, where $i\neq j$ and $i,j\in \{1,\ldots,t\}$;

\item for each vertex $i$ of $T$, there is a finite group $B_i$;

\item for each edge $(i,j)$ of $T$, there is a group $A_{ij}$, which is a subgroup of both the groups $B_i$ and $B_j$.
\end{itemize}
Then we define
$$G:=\pi_1(B(-),T),$$
i.e. $G$ is the appropriate amalgamated free product, as described in \cite[Example I.3.5(vi)]{Dicks}. We define the sequence $\rho$ as the concatenation of $B_1,\ldots,B_t$ amalgamated along the subsets $A_{ij}$.
\begin{example}
If $T$ is the tree with two vertices $1,2$ and one edge $(1,2)$, then we get
$$\pi_1(B(-),T)=B_1\ast_{A_{12}}B_2,$$
as at the beginning of this subsection.
\end{example}
For each vertex $i$ in $T$, let $\rho_i$ be the subsequence of $\rho$ corresponding to $B_i$. We assume that $(K,\rho)$ is a $G$-field and $(V,W)$ are varieties such that $W\subseteq {}^{\rho}V$. Then for each vertex $i$ in $T$, let $W_i$ be the (Zariski closure of the) projection of $W$ on ${}^{\rho_i}V$.
\begin{definition}\label{defamgen}
We say that $(V,W)$ is a \emph{$\pi_1(B(-),T)$-pair}, if for each vertex $i$ in $T$, $(V,W_i)$ is a $B_i$-pair.
\end{definition}
The proof of the next result is basically the same as the proof of Prop. \ref{g1astg2}, and the definition of the variety $W'$ is the same as well.
\begin{prop}[$\pi_1(B(-),T)$-Prolongation Lemma]\label{treeprol}
If $(V,W)$ is a $\pi_1(B(-),T)$-pair, then $(W,W')$ is a $\pi_1(B(-),T)$-pair as well.
\end{prop}
\begin{proof}
For each vertex $i$ in $T$, we need to check that $(W,W'_i)$ is a $B_i$-pair. It can be proved in exactly the same way as it is shown in the proof of Prop. \ref{g1astg2} that $(W,W'_B)$ is a $B$-pair.
\end{proof}

We postpone stating the conclusion about the theory $\gtcf$ given by Theorem \ref{algproofgen} till we consider the general case in Section \ref{sechnngen}.

\subsection{HNN-extensions}\label{sechnn}
We start from a special case which will serve our intuitions in a similar way as Example \ref{simpleamal} did in Section \ref{secampr}. Let $C_2\times C_2=\{1,\sigma,\tau,\gamma\}$ and consider the following HNN-extension:
$$\alpha:\{1,\sigma\}\cong \{1,\tau\},\ \ \ G:=(C_2\times C_2)\ast_{\alpha}.$$
Then the crucial relation defining $G$ is $\sigma t=t\tau$. We take:
$$\rho:=(1,\sigma,\tau,\gamma,t,t\sigma,t\tau,t\gamma),$$
$$\rho\rho=\left[\begin{array}{cccc|cccc}
1             & \sigma        & \tau          & \gamma   & t    & t\sigma   & t\tau     & t\gamma   \\
\sigma       &   1       & \gamma             & \tau   & t\tau       & t\gamma  & t       & t\sigma   \\
\hline
\tau         &   \gamma         & 1        & \sigma    &  \tau t       & \tau t\sigma   & \tau t \tau            & \tau t \gamma   \\
 \gamma      &   \tau         & \sigma          & 1   &    \tau t \tau     &  \tau t \gamma   & \tau t          & \tau t\sigma    \\
\hline
\hline
t           & t\sigma    & t\tau     & t\gamma   & t^2      & t^2\sigma   & t^2\tau     & t^2\gamma \\
t\sigma   &   t      & t\gamma         & t\tau      &     t^2\tau   &  t^2\gamma  & t^2        & t^2\sigma \\
\hline
t\tau       &   t\gamma     & t         & t\sigma      &  t\tau t      & t\tau t\sigma   & t\tau t \tau          & t\tau t \gamma  \\
t\gamma      &  t \tau     & t\sigma             & t      &   t \tau t \tau    & t \tau t \gamma   & t\tau t         & t\tau t\sigma
\end{array}\right].$$
We check below some easy calculations related to the entries of the matrix above:
$$\sigma(t\tau)=t\tau\tau=t,\ \ \ \ \ \ \gamma(t\sigma)=\tau\sigma t\sigma=\tau t\tau\sigma=\tau t\gamma.$$
For $W\subseteq {}^{\rho}V$ and a subsequence $\rho'\subseteq \rho$, we define the following coordinate projection and the corresponding image of $W$:
$$\pi_{\rho'}:{}^{\rho}V\to {}^{\rho'}V,\ \ \ W_{\rho'}:=\pi_{\rho'}(W).$$
We fix the following sequences:
$$\rho_0:=(1,\sigma,\tau,\gamma),\ \ \ t\rho_0:=(t,t\sigma ,t\tau t,t\gamma ),$$
and we define below our crucial notion in this special case.
\begin{definition}\label{c2c2}
$(V,W)$ is a \emph{$(C_2\times C_2)\ast_{\alpha}$-pair}, if:
\begin{enumerate}
\item ${}^{t}W_{\rho_0}=W_{t\rho_0}$.

\item $(V,W_{\rho_0})$ is a $(C_2\times C_2)$-pair;
\end{enumerate}
\end{definition}
\begin{remark}\label{imply}
The item $(1)$ means that $(W_{\rho_0},W)$ is a $\Zz$-pair. The items $(1)$ and $(2)$ \emph{imply} that all the projections $W\to {}^{\rho_i}V$ are dominant, since for $i\in \rho_0$ the dominance follows from the item $(2)$ and for $i\in t\rho_0$ it follows from the item $(1)$.
\end{remark}
To prove the $G$-Prolongation Lemma in this context, the following observation is useful:
$${}^{\sigma}({}^{\rho}V)={}^{\lambda(\rho)}V,$$
where
$$\lambda:=(12)(34)(57)(68)\in S_8$$
is the appropriate permutation, which is applied to the $8$-tuple $\rho$ in the obvious way. We also consider $\lambda$ as the following permutation isomorphism $\lambda_V$ (defined using the $\lambda_V^i$-maps from Section \ref{n0} for $G_0=C_2\times C_2$):
$$\lambda_V:=\lambda^{\sigma}_V\times {}^{t}\left(\lambda^{\tau}_{V}\right):{}^{\rho}V={}^{\rho_0}V\times {}^{t\rho_0}V
\to {}^{\sigma \rho_0}V\times {}^{t\tau \rho_0}V={}^{\sigma\rho}V.$$
Suppose $(V,W)$ is a $G$-pair and $W':={}^{\rho}W\cap {}^{\rho\cdot\rho}V$. To have a convenient description
of $W'_{\rho_0}$, we define:
\\
$$\widetilde{\lambda}_V:=\left(\id_V,\lambda_V\right):{}^{\rho}V\to {}^{\rho}V\times {}^{\sigma}{}^{\rho}V.$$
Analyzing the above matrix of $\rho\rho$, it is easy to see that:
$$W'_{\rho_0}=\widetilde{\lambda}_V(W)\times_{({}^{\rho_0}V\times {}^{\rho_0}V)}\widetilde{\lambda}_{{}^{\tau}V}({}^{\tau}W),$$
where the morphisms
$$\widetilde{\lambda}_V(W)\to {}^{\rho_0}V\times {}^{\rho_0}V,\ \ \ \widetilde{\lambda}_{{}^{\tau}V}({}^{\tau}W)\to {}^{\rho_0}V\times {}^{\rho_0}V,$$
which we need to define the fiber product above, come from the usual projection map composed with appropriate coordinate permutations.
From this description of $W'_{\rho_0}$, the $G$-Prolongation Lemma follows, as we will see soon in the general case of an HNN-extension of a finite group.
\\
\\
Let us consider now the general situation, where $G_0=\{g_1=1,\ldots,g_e\}$ is finite and $\alpha:A\to B$ is an isomorphism of subgroups of $G_0$. Let
$$G:=G_0\ast_{\alpha},\ \ \ \rho:=(g_1,\ldots,g_e,tg_1,\ldots,tg_e),\ \ \ \rho_0:=(g_1,\ldots,g_e).$$
For each $\sigma\in A$, we have the following equality in $G$:
$$\sigma t=t\alpha(\sigma).$$
Let us fix a $G$-field $(K,\rho)$ and a pair of varieties $(V,W)$ such that $W\subseteq {}^{\rho}V$. For any $g\in G_0$, let
$$\lambda^g_V:{}^{\rho_0}V\to {}^{g\rho_0}V$$
be the obvious permutation map (appearing in Section \ref{n0}), where $\rho_0$ is the subsequence of $\rho$ corresponding to $G_0$.
For any $\sigma \in A$, we also define (similarly as in the case of $G=(C_2\times C_2)\ast_{\alpha}$)
$$\widetilde{\lambda}^{\sigma }_V:{}^{\rho}V\to {}^{\sigma \rho}V,\ \ \ \
\widetilde{\lambda}^{\sigma}_V:=\lambda^{\sigma}_V\times {}^{t}\left(\lambda^{\alpha(\sigma)}_{V}\right).$$
We aim to find a good notion of a $G$-pair using the above choice of the sequence of generators $\rho$. We define the coordinate projections similarly as in the case of $G=(C_2\times C_2)\ast_{\alpha}$.
\begin{definition}\label{defhnn}
$(V,W)$ is a \emph{$G$-pair}, if:
\begin{enumerate}
\item ${}^{t}W_{\rho_0}=W_{t\rho_0}$;

\item $(V,W_{\rho_0})$ is a $G_0$-pair;
\end{enumerate}
\end{definition}
\begin{remark}
Similarly as in Remark \ref{imply}, the conditions $(1)$ and $(2)$ imply that all the projections $W\to {}^{\rho_i}V$ are dominant.
\end{remark}
\begin{example}\label{hnnex}
We check some special cases of HNN-extensions.
\begin{enumerate}
\item We have
$$\{1\}\ast_{\{1\}}\cong \Zz,\ \ \ \rho=(1,t)$$
and the axioms for $\{1\}\ast_{\{1\}}-\tcf$ coincide with the classical axioms for ACFA.

\item We have
$$G\ast_{\{1\}}\cong G\ast \Zz,\ \ \ \rho=(g_1,\ldots,g_e,tg_1,\ldots,tg_e).$$
The choice of $\rho$ here is ``less economical'' than the choice from Section \ref{secampr}, where we had $\rho=(g_1,\ldots,g_e,t)$.
Hence the axioms for $G\ast_{\{1\}}-\tcf$ have a slightly different form than the equivalent axioms for $(G\ast \Zz)-\tcf$ from Section \ref{secampr}.

\item We have
$$G\ast_{\id_G}\cong G\times \Zz,\ \ \  \rho=(g_1,\ldots,g_e,tg_1,\ldots,tg_e)$$
and the axioms for $G\ast_{\id_G}-\tcf$ are implied by the axioms for $(G\times \Zz)-\tcf$ from Section \ref{secgen1}. (Obviously, since both the theories have the same models, both sets of axioms are equivalent.)
\end{enumerate}
\end{example}
We recall that we have fixed a $G$-field $(K,\rho)$ and a pair of varieties $(V,W)$ such that $W\subseteq {}^{\rho}V$. We also define as usual:
$$W':={}^{\rho}W\cap {}^{\rho\cdot\rho}V.$$
\begin{prop}[$G$-Prolongation Lemma for HNN-extensions]\label{prolnhh}
If $(V,W)$ is a $G$-pair, then $(W,W')$ is a $G$-pair.
\end{prop}
\begin{proof}
We start from a convenient description of the projections $W'_{\rho_0},W'_{t{\rho_0}}$. Let $A=\{1=\sigma_1,\sigma_2,\ldots,\sigma_m\}$.
We define first:
$$\widetilde{\lambda}_V:=\left(\id_V,\widetilde{\lambda}^{\sigma_2}_V,\ldots,\widetilde{\lambda}^{\sigma_m}_V\right):{}^{\rho}V\to {}^{A}\left({}^{\rho}V\right)
={}^{\rho}V\times {}^{\sigma_2\rho}V\times \ldots \times{}^{\sigma_m\rho}V.$$
Let $\{1,\tau_2,\ldots,\tau_k\}$ be a set of representatives of the cosets in $G_0/A$. The corresponding matrix $\rho\rho$ here has a similar shape as the one in the case of $G=(C_2\times C_2)\ast_{\alpha}$ (considered in the beginning of this subsection), and we get similar conclusions:
$$W'_{\rho_0}=\widetilde{\lambda}_V(W)\times_{({}^{{\rho_0}}V\times {}^{{\rho_0}}V)}\widetilde{\lambda}_{{}^{\tau_2}V}({}^{\tau_2}W)\times_{({}^{{\rho_0}}V\times {}^{{\rho_0}}V)} \ldots \times_{({}^{{\rho_0}}V\times {}^{{\rho_0}}V)}\widetilde{\lambda}_{{}^{\tau_k}V}({}^{\tau_k}W),$$
$$W'_{t{\rho_0}}=\widetilde{\lambda}_{{}^{t}V}({}^{t}W)\times_{({}^{t{\rho_0}}V\times {}^{t{\rho_0}}V)}\widetilde{\lambda}_{{}^{t\tau_2}V}({}^{t\tau_2}W)\times_{({}^{{t\rho_0}}V\times {}^{{t\rho_0}}V)} \ldots \times_{({}^{{t\rho_0}}V\times {}^{{t\rho_0}}V)}\widetilde{\lambda}_{{}^{t\tau_k}V}({}^{t\tau_k}W),$$
where (as in the case of $G=(C_2\times C_2)\ast_{\alpha}$ discussed above) the morphisms
$$\widetilde{\lambda}_V(W)\to {}^{\rho_0}V\times {}^{\rho_0}V,\ \ \ \widetilde{\lambda}_{{}^{\tau_i}V}({}^{\tau_i}W)\to {}^{\rho_0}V\times {}^{\rho_0}V$$
come from the projection map composed with appropriate coordinate permutations.

Hence, we immediately get that
$${}^{t}W'_{{\rho_0}}=W'_{t{\rho_0}},$$
so the condition $(1)$ from the definition of a $G$-pair holds.

For the condition $(2)$, we take $g\in G_0$ and calculate:
\begin{IEEEeqnarray*}{rCl}
{}^{g}\left(W'_{{\rho_0}}\right) &=& {}^{g}\left(\widetilde{\lambda}_V(W)\times_{({}^{{\rho_0}}V\times {}^{{\rho_0}}V)}\widetilde{\lambda}_{{}^{\tau_2}V}({}^{\tau_2}W)
\times_{({}^{{\rho_0}}V\times {}^{{\rho_0}}V)} \ldots \times_{({}^{{\rho_0}}V\times {}^{{\rho_0}}V)}\widetilde{\lambda}_{{}^{\tau_k}V}({}^{\tau_k}W)\right)\\
 &=& \widetilde{\lambda}_{{}^{g}V}({}^{g}W)\times_{({}^{g {\rho_0}}V\times {}^{g {\rho_0}}V)}\widetilde{\lambda}_{{}^{g\tau_2}V}({}^{g \tau_2}W)\times_{({}^{g{\rho_0}}V\times {}^{g{\rho_0}}V)}\ldots \times_{({}^{g {\rho_0}}V\times {}^{g {\rho_0}}V)}\widetilde{\lambda}_{{}^{g\tau_k}V}({}^{g \tau_k}W)\\
& = & \lambda_W^{g}\left(\widetilde{\lambda}_V(W)\times_{({}^{{\rho_0}}V\times {}^{{\rho_0}}V)}\widetilde{\lambda}_{{}^{\tau_2}V}({}^{\tau_2}W)
\times_{({}^{{\rho_0}}V\times {}^{{\rho_0}}V)} \ldots \times_{({}^{{\rho_0}}V\times {}^{{\rho_0}}V)}\widetilde{\lambda}_{{}^{\tau_k}V}({}^{\tau_k}W)\right)\\
& = & \lambda_W^{g}\left(W_{\rho_0}'\right),
\end{IEEEeqnarray*}
which is what we wanted.
\end{proof}
Again, we postpone stating the conclusions about the theory $\gtcf$ till Section \ref{sechnngen}.

\subsection{General case}\label{sechnngen}
To cover the general case of an arbitrary (finitely generated) virtually free group, we will use the Bass-Serre theory concerning groups acting by automorphisms on  trees. We will not explain the notion of the \emph{fundamental group of a graph of groups} appearing in Theorem \ref{bs1}, since it is immediately clarified in Theorem \ref{bs2} using the notions which have been introduced already.
\begin{theorem}[p. 104, Corollary IV.1.9 in \cite{Dicks}]\label{bs1}
A group is virtually free if and only if, it is the fundamental group of a finite graph of finite groups.
\end{theorem}
\begin{theorem}[p. 14, Example I.3.5(vi) in \cite{Dicks}]\label{bs2}
The fundamental group of a graph of groups can be obtained by successively performing:
\begin{itemize}
\item one free product with amalgamation for each edge in the maximal subtree;

\item and then one HNN extension for each edge not in the maximal subtree.
\end{itemize}
\end{theorem}
Let $G$ be a finitely generated, virtually free group. By Theorem \ref{bs1}, $G$ may be represented as the fundamental group of a finite graph of finite groups, i.e.
$$G=\pi_1(G(-),Y).$$
By Theorem \ref{bs2}, there is a maximal subtree $T$ of $Y$ such that $G$ comes from a sequence of HNN-extensions of $\pi_1(G(-),T)$ and we know that
$\pi_1(G(-),T)$ is the amalgamated free product along the tree $T$ as in Section \ref{secampr}. Let $VY$ denote the set of vertices of $Y$, $EY$ the set of edges of $Y$ and we set $E:=EY\setminus ET$. Then each HNN-extension, which we need to obtain the group $G$ from the group $\pi_1(G(-),T)$, comes from an edge $e\in E$.

For any $i\in VY$, let $G_i$ denote the corresponding vertex groups and for any $e\in E$ let
$$\alpha_e:A_{\overline{\iota}e}\cong A_{\overline{\tau}e}$$
be the corresponding edge group isomorphism, where $\overline{\iota}e$ (resp. $\overline{\tau}e$) is the initial (resp. terminal)  vertex of the edge $e$,  $A_{\overline{\iota}e}\leqslant G_{\overline{\iota}e}$ and $A_{\overline{\tau}e}\leqslant G_{\overline{\tau}e}$.

Let $\rho_T$ be the appropriate amalgamated union of the underlying sets of the vertex groups of $T$ (so also of $Y$) as in Section \ref{secampr}.
For any $e\in E$, we introduce a new variable $t_e$.
We define below our fixed sequence of generators of $G$:
$$\rho:=\rho_T\cup \bigcup_{e\in E}t_e\rho_T.$$
Note that (using the definition of the fundamental group of a graph of groups) Assumption \ref{assume} is satisfied for the marked group $(G,\rho)$.

For any $i\in VY$, let $\rho^i$ denote the subsequence of $\rho_T$ corresponding to $G_i$. Assume now we have a $G$-field $(K,\rho)$. We take a pair of varieties $(V,W)$ such that $W\subseteq {}^{\rho}V$. For any $i\in VY,e\in E$, we denote by $W_i$ the (Zariski closure of the) projection of $W$ to ${}^{\rho_{i}}V$, and by $W_{e,i}$ the (Zariski closure of the) projection of $W$ to ${}^{t_e\rho_{i}}V$. For any subsequence $\rho'\subseteq \rho$, we sometimes also denote by $W_{\rho'}$ the (Zariski closure of the) projection of $W$ to ${}^{\rho'}V$.

To find the right definition of a $G$-pair in this general context, we consider below one relatively easy example which does not fit into the classes of groups considered in Sections \ref{secampr}, \ref{sechnn}.
\begin{example}\label{exgen}
We define the graph $Y$ in the following way:
$$VY=\{1,2\},\ EY=\{e_1,e\},\ \overline{\iota}e_1=1,\ \overline{\tau}e_1=2,\ \overline{\iota}e=1,\ \overline{\tau}e=2.$$
So, $Y$ consists of two vertices and two edges between these vertices (going into the same direction). Let $T$ be a maximal tree with the edge $e_1$, so $E=\{e\}$. We put the structure of a graph of groups on $Y$ as follows:
$$G_1=C_4=\{1,\alpha,\beta,\gamma\},\ \ \ G_2=C_4=\{1,\alpha',\beta',\gamma'\},\ \ \ A_{e_1}=\{1\},\ \ \ A_{e}=C_2,$$
and the isomorphism $\alpha_e$, corresponding to the edge $e$, maps $\beta$ to $\beta'$. Hence we have:
$$\pi_1(G(-),Y)=\left(C_4\ast C_4\right)\ast_{\alpha_e},$$
where the crucial relation is $\beta t=t\beta'$ (to ease the notation, we write $t$ for $t_e$ in this example). We obtain the following sequences:
$$\rho_T=(1,\alpha,\beta,\gamma,\alpha',\beta',\gamma'),\ \ \ \ \ \ \rho=(1,\alpha,\beta,\gamma,\alpha',\beta',\gamma',t,t\alpha,t\beta,t\gamma,t\alpha',t\beta',t\gamma').$$
We need to calculate the matrix $\rho\rho$ which will serve as just one, and relatively small, block of the potentially huge matrix $\rho\rho$, we need to deal with in the general case of a virtually free group $G$.
$$\left[\begin{array}{c|c|c|c|c}
(1,\alpha,\beta,\gamma) & (\alpha',\beta',\gamma') & t   & t(\alpha,\beta,\gamma) & t(\alpha',\beta',\gamma')  \\
(\alpha,\beta,\gamma,1) & \alpha(\alpha',\beta',\gamma') & \alpha t   &\alpha t(\alpha,\beta,\gamma) & \alpha t(\alpha',\beta',\gamma')   \\
(\beta,\gamma,1,\alpha)  & \beta(\alpha',\beta',\gamma') & t\beta'   & t\beta'(\alpha,\beta,\gamma) & t(\gamma',1,\alpha')  \\
(\gamma,1,\alpha,\beta) & \gamma(\alpha',\beta',\gamma') & \alpha t\beta' &\alpha t\beta'(\alpha,\beta,\gamma) &\alpha t(\gamma',1,\alpha') \\
\hline
\alpha'(1,\alpha,\beta,\gamma) & (\beta',\gamma',1) & \alpha' t   &\alpha' t(\alpha,\beta,\gamma) & \alpha' t(\alpha',\beta',\gamma')   \\
\beta'(1,\alpha,\beta,\gamma)  & (\gamma',1,\alpha') & \beta' t   & \beta' t(\alpha,\beta,\gamma) & \beta' t(\alpha',\beta',\gamma') \\
\gamma'(1,\alpha,\beta,\gamma) & (1,\alpha',\beta') & \gamma' t & \gamma' t(\alpha,\beta,\gamma) &\gamma' t(\alpha',\beta',\gamma')\\
\hline
t(1,\alpha,\beta,\gamma) & t(\alpha',\beta',\gamma') & t^2   & t^2(\alpha,\beta,\gamma) & t^2(\alpha',\beta',\gamma')  \\
\hline
t(\alpha,\beta,\gamma,1) & t\alpha(\alpha',\beta',\gamma') & t\alpha t   &t\alpha t(\alpha,\beta,\gamma) & t\alpha t(\alpha',\beta',\gamma')   \\
t(\beta,\gamma,1,\alpha)  & t\beta(\alpha',\beta',\gamma') & t^2\beta'   & t^2\beta'(\alpha,\beta,\gamma) & t^2(\gamma',1,\alpha')  \\
t(\gamma,1,\alpha,\beta) & t\gamma(\alpha',\beta',\gamma') & t\alpha t\beta' &t\alpha t\beta'(\alpha,\beta,\gamma) &t\alpha t(\gamma',1,\alpha') \\
\hline
t\alpha'(1,\alpha,\beta,\gamma) & t(\beta',\gamma',1) & t\alpha' t   &t\alpha' t(\alpha,\beta,\gamma) & t\alpha' t(\alpha',\beta',\gamma')   \\
t\beta'(1,\alpha,\beta,\gamma)  & t(\gamma',1,\alpha') & t\beta' t   & t\beta' t(\alpha,\beta,\gamma) & t\beta' t(\alpha',\beta',\gamma') \\
t\gamma'(1,\alpha,\beta,\gamma) & t(1,\alpha',\beta') & t\gamma' t & t\gamma' t(\alpha,\beta,\gamma) &t\gamma' t(\alpha',\beta',\gamma')
\end{array}\right]$$
The definition of a $(C_4\ast C_4)\ast_{\alpha_e}$-pair below combines Definitions \ref{def1am}, \ref{c2c2}.
\begin{definition}\label{defhnn}
$(V,W)$ is a \emph{$(C_4\ast C_4)\ast_{\alpha_e}$-pair}, if:
\begin{enumerate}
\item ${}^{t}W_{\rho_T}=W_{t\rho_T}$;

\item $(V,W_{\rho_T})$ is a $(C_4\ast C_4)$-pair (in the sense of Definition \ref{def1am});

\end{enumerate}
\end{definition}

\end{example}
%
Using the intuitions coming from Example \ref{exgen} and Definition \ref{defhnn}, we give the general definition below.
\begin{definition}
We say that $(V,W)$ is a \emph{$G$-pair}, if:
\begin{enumerate}
\item for all $e\in E$ and $i\in VY$, we have
$${}^{t_{e}}W_{i}=W_{e,i}.$$


\item $(V,W_{\rho_T})$ is a $\pi_1(G(-),T)$-pair (in the sense of Definition \ref{defamgen});

\end{enumerate}
\end{definition}
\begin{remark}
Note that again (as in Remark \ref{imply}), the conditions $(1)$ and $(2)$ \emph{imply} that all the projections $W\to {}^{\rho_i}V$ are dominant.
\end{remark}
As usual, for any pair of varieties $(V,W)$ such that $W\subseteq {}^{\rho}V$, we define
$$W':={}^{\rho}W\cap {}^{\rho\cdot\rho}V.$$
\begin{prop}[$G$-Prolongation Lemma for virtually free $G$]\label{prolmain}
If $G$ is a finitely generated, virtually free group and $(V,W)$ is a $G$-pair, then $(W,W')$ is a $G$-pair.
\end{prop}
\begin{proof}
The proof of the condition $(1)$ goes in the same way as in the proof of Prop. \ref{prolnhh}. We give few details below. For each $i\in VY$, we have:
$$W'_i=\widetilde{\lambda}_V(W)\times_{\left({}^{\rho^i}V\times {}^{\rho^i}V\right)}\widetilde{\lambda}_{{}^{\tau_2}V}({}^{\tau_2}W)\times_{\left({}^{\rho^i}V\times {}^{\rho^i}V\right)} \ldots
\times_{\left({}^{\rho^i}V\times {}^{\rho^i}V\right)}\widetilde{\lambda}_{{}^{\tau_k}V}({}^{\tau_k}W),$$
where $\{1,\tau_2,\ldots,\tau_k\}$ are representatives of the cosets in $G_i/A_i$ and the natural map $\widetilde{\lambda}_V$ is defined using appropriate coordinate permutations as as in the proof of Prop. \ref{prolnhh}. Using a similar description for the variety $W'_{i,e}$, we show the condition $(1)$ as in the proof of Prop. \ref{prolnhh}.

The proof of the condition $(2)$ goes in the same way as the proof of Prop. \ref{treeprol} (or rather Prop. \ref{g1astg2}), and we leave the details to the reader.
\end{proof}
\begin{prop}\label{secmain}
If $\left(K(a),K(\rho'(a)),K(\rho''(\rho'(a)))\right)$ is a $G$-kernel (as in Definition \ref{kerdefg}), then
$$\left(\locus_K(a),\locus_K(\rho'(a))\right)$$
is a $G$-pair.
\end{prop}
\begin{proof}
This can be proved as Lemma \ref{sstrong0} using an obvious analogue of Lemma \ref{loci}.
\end{proof}
Prop. \ref{prolmain} and Prop. \ref{secmain} give us the assumptions from Theorem \ref{algproofgen}, hence Theorem \ref{algproofgen} provides the main result of this paper.
\begin{theorem}\label{mainthm}
If $G$ is a finitely generated virtually free group, then the theory $\gtcf$ exists. Moreover, the axioms of $\gtcf$ are as in Theorem \ref{algproofgen}.
\end{theorem}

\subsection{Semi-direct products}\label{secgen1}
In this section, we present explicit (i.e. obtained without a usage of the Bass-Serre theory) axioms for the theory $\gtcf$ where $G$ is of the form $G=F_n\rtimes G_0$ for some types of actions of a finite group $G_0$ on the free group $F_n$.

We start from explicit axioms for the theory $(F_n\times G_0)-\tcf$. We present:
$$F_n\times G_0=\langle 1,\sigma_1,\ldots,\sigma_n\rangle \times \{g_1=1,\ldots,g_e\}=\langle \bar{\sigma}\rangle\times \langle \bar{g}\rangle,$$
and for the sequence of generators we take the following:
$$\rho:=\bar{\sigma}\bar{g}=(g_1=1,\sigma_1,\ldots,\sigma_n;g_2,\sigma_1g_2,\ldots,\sigma_ng_2;\ldots;g_e,\sigma_1g_e,\ldots,\sigma_ng_e).$$
Then the crucial product matrix $\rho\rho$ has the shape of the \emph{Kronecker product} (or the \emph{tensor product}) of matrices $\bar{\sigma}\bar{\sigma}$ and $\bar{g}\bar{g}$.
$$\rho\rho=\left[\begin{array}{cccc|cc|cccc}
1          & \sigma_1       & \ldots   & \sigma_n  & \ldots & \ldots & g_e & \sigma_1g_e  & \ldots  & \sigma_ng_e  \\
\sigma_1   & \sigma_1^2     & \ldots  & \sigma_1\sigma_n & \ldots & \ldots & \sigma_1g_e & \sigma_1^2g_e  & \ldots  & \sigma_1\sigma_ng_e  \\
  &    & \ddots  &   & \ddots & \ddots &  &   & \ddots  &   \\
\sigma_n   & \sigma_n\sigma_1     & \ldots   & \sigma_n^2 & \ldots & \ldots & \sigma_ng_e & \sigma_n\sigma_1g_e  & \ldots  & \sigma_n^2g_e  \\
\hline
  &    & \ddots  &   & \ddots & \ddots &  &   & \ddots  &   \\
\hline
  g_e          & \sigma_1g_e      & \ldots & \sigma_ng_e   & \ldots & \ldots & g_e^2 & \sigma_1g_e^2  & \ldots  & \sigma_ng_e^2  \\
\sigma_1g_e   & \sigma_1^2g_e     & \ldots & \sigma_1\sigma_ng_e  & \ldots  & \ldots & \sigma_1g_e^2 & \sigma_1^2g_e^2  & \ldots  & \sigma_1\sigma_ng_e^2  \\
  &    & \ddots  &   & \ddots & \ddots &  &   & \ddots  &   \\
\sigma_ng_e   & \sigma_n\sigma_1g_e     & \ldots & \sigma_n^2g_e &  \ldots  & \ldots & \sigma_ng_e^2 & \sigma_n\sigma_1g_e^2  & \ldots  & \sigma_n^2g_e^2
\end{array}\right].$$
Clearly, the Word Problem for the marked group $(G,\rho)$ is of the right form (i.e. it satisfies Assumption \ref{assume}):
$$\sigma_jg_i=g_i\sigma_j,\ \ \ g_ig_j=g_k.$$
For $W\subseteq {}^{\rho}V$, we need to find the appropriate iterativity conditions. They will come from both the $F_n$-iterativity conditions and the $G_0$-iterativity conditions. We define first the projection maps:
$$\pi_{\bar{\sigma}}:{}^{\rho}V\to {}^{\bar{\sigma}}V,\ \ \pi_{\bar{g}}:{}^{\rho}V\to {}^{{\bar{g}}}V.$$
We define now the auxiliary varieties (everything up to Zariski closure):
$$W_{\bar{\sigma}}:=\pi_{\bar{\sigma}}(W),\ \ \ W_{\bar{g}}:=\pi_{\bar{g}}(W).$$
We say that $(V,W)$ is a \emph{$G$-pair}, if:
\begin{enumerate}
\item $W$ projects dominantly on $V$;

\item $(W_{\bar{g}},W)$ is an $F_n$-pair;

\item $(W_{\bar{\sigma}},W)$ is a $G_0$-pair.
\end{enumerate}
\begin{remark}
\begin{enumerate}
\item If $F_n$ is trivial, then we get the known axioms of $G_0-\tcf$, and if $G_0$ is trivial we get the known axioms of ACFA$_n$.
\item If $G=\Zz\times G_0$, then we are in the situation from Example \ref{hnnex}(3).
\end{enumerate}
\end{remark}
After showing the appropriate prolongation lemma and defining the theory $(F_n\times G_0)-\tcf$ as in Theorem \ref{algproofgen}, we obtain the following.
\begin{theorem}\label{thmprod}
The theory $(F_n\times G_0)-\tcf$ axiomatizes the class of existentially closed $(F_n\times G_0)$-fields.
\end{theorem}
We illustrate the semi-direct product case using the following example:
$$D_{\infty}=\Zz\rtimes C_2=\langle\rho_{\sigma}\rangle \rtimes \langle \rho_{\tau} \rangle:=\langle\sigma^{-1},1,\sigma \rangle \rtimes \langle 1,\tau \rangle,$$
$$\rho=\rho_{\sigma}\rho_{\tau}=(\sigma^{-1},1,\sigma,\sigma^{-1}\tau,\tau,\sigma\tau),$$
$$\rho\rho=\left[\begin{array}{ccc|ccc}
\sigma^{-2}  & \sigma^{-1} & 1 &   \sigma^{-2}\tau  & \sigma^{-1}\tau & \tau   \\
\sigma^{-1}  &  1          & \sigma &    \sigma^{-1}\tau  &  \tau   & \sigma\tau   \\
1            & \sigma      & \sigma^2 &  \tau       &   \sigma\tau      &  \sigma^2\tau  \\
\hline
 \tau      & \sigma^{-1}\tau      & \sigma^{-2}\tau & 1  & \sigma^{-1} & \sigma^{-2}  \\
 \sigma\tau      & \tau      & \sigma^{-1}\tau & \sigma  & 1 & \sigma^{-1}  \\
 \sigma^2\tau      & \sigma\tau      & \tau & \sigma^2  & \sigma & 1
\end{array}\right]
=\left[\begin{array}{c|c}
\rho_{\sigma}\rho_{\sigma}  &  \rho_{\sigma}\rho_{\sigma}\cdot \tau  \\
\hline
\overline{\rho_{\sigma}\rho_{\sigma}}\cdot \tau    & \overline{\rho_{\sigma}\rho_{\sigma}}
\end{array}\right],$$
where the  matrix $\overline{\rho_{\sigma}\rho_{\sigma}}$ is obtained from the matrix $\rho_{\sigma}\rho_{\sigma}$ by interchanging the first column with the third column. Hence the above matrix can be though of as a ``twisted Kronecker product''.
\\
Assume that $W\subseteq {}^{\rho}V$. We define again the appropriate projection maps:
$$\pi_{\sigma}=\pi_{\sigma}^V:{}^{\rho}V\to {}^{\rho_{\sigma}}V,\ \ \pi_{\tau}=\pi_{\tau}^V:{}^{\rho}V\to {}^{\rho_{\tau}}V,$$
and the auxiliary varieties (everything up to Zariski closure):
$$W_{\sigma}:=\pi_{\sigma}(W),\ \ \ W_{\tau}:=\pi_{\tau}(W).$$
Before defining the notion of a $D_{\infty}$-pair, we notice the following:
\begin{IEEEeqnarray*}{rCl}
{}^{\rho}V & = & {}^{\sigma^{-1}}V\times V\times {}^{\sigma}V\times {}^{\sigma^{-1}\tau}V\times {}^{\tau}V\times {}^{\sigma\tau}V \\
 & = & {}^{\sigma^{-1}}V\times V\times {}^{\sigma}V\times {}^{\tau\sigma}V\times {}^{\tau}V\times {}^{\tau\sigma^{-1}}V\\
 &= & \mathrm{tw}^4_6\left({}^{\rho_{\sigma}}V\times {}^{{\tau}}({}^{\rho_{\sigma}}V)\right),
\end{IEEEeqnarray*}
where the map $\mathrm{tw}^4_6$ exchanges the fourth coordinate with the sixth coordinate. In particular, we get
$$\mathrm{tw}^4_6(W)\subseteq {}^{\rho_{\sigma}}V\times {}^{{\tau}}({}^{\rho_{\sigma}}V).$$
We say that $(V,W)$ is a \emph{$D_{\infty}$-pair}, if (the notion of a \emph{$(\Zz,\rho_{\sigma})$-pair} below should be easy to guess):
\begin{enumerate}
\item $W$ projects dominantly on $V$;

\item $(W_{\sigma},\mathrm{tw}^4_6(W))$ is a $C_2$-pair;

\item $(W_{\tau},W)$ is a $(\Zz,\rho_{\sigma})$-pair.
\end{enumerate}
We discuss now the general case of a semi-direct product. Let us fix a free basis $\{x_1,\ldots,x_n\}$ of $F_n$. By a theorem of Nielsen \cite{Niel}, the group $\aut(F_n)$ is generated by automorphisms belonging to the following three basic types (we merge the first ``classical'' two types into one type):
\begin{enumerate}
\item induced by a permutation of $\{x_1,\ldots,x_n\}$;

\item induced by a map $x_i\mapsto x_i^{-1}$;

\item induced by a map $x_i\mapsto x_ix_j$ for $i\neq j$.
\end{enumerate}
We can treat the actions coming from the type $(2)$ in a similar way as in the case of the group $D_{\infty}$. We can also deal with the case of the type $(1)$ automorphisms of $F_n$ using row permutations on the top of (already visible in the case of $D_{\infty}$) column permutations. However, it is not clear what to do with the type $(3)$ actions. A test case is the group $G=F_2\rtimes C_3$, where the corresponding automorphism of $F_2$ of order 3 is given by:
$$x_1\mapsto x_1^{-1}x_2,\ \ x_2\mapsto x_1^{-1}.$$
For any possible choice of a sequence of generators $\rho$, the matrix $\rho\rho$ is not going to be a ``twisted Kronecker product'' anymore, so in this case we have only the procedure coming from the Basse-Serre theory as explained in Section \ref{sechnngen}.

\section{Model-theoretic properties}\label{secmt}
In this section, we describe the model-theoretic properties of the theories obtained in Theorem \ref{mainthm}. In short, all the new theories do not fit nicely into the (neo-)stability hierarchy, i.e. they are not NTP$_2$ (Theorem \ref{nonntp}).

We recall the necessary notions from the theory of profinite groups. For a field $L$, we denote by $\gal(L)$ the \emph{absolute Galois group} of $L$. For a discrete group $G$, we denote by $\widehat{G}$ its \emph{profinite completion} and by $\widehat{G}(p)$ its \emph{$p$-profinite completion} (see \cite[Remark 17.4.7]{FrJa}). For a profinite group $H$, we denote by $\widetilde{H}\to H$ the \emph{universal Frattini cover} of $H$ (see \cite[ Prop. 22.6.1]{FrJa}). The notation $A\leqslant_cB$ means that $A$ is a closed  subgroup of $B$. A profinite group $H$ is \emph{small}, if for any $n>0$, there are finitely many open subgroups of $H$ of index $n$ (so, a field $K$ is \emph{bounded} if and only if $\gal(K)$ is small).

\subsection{General properties}\label{secgenpr}
Let us fix a marked group $(G,\rho)$, where the fixed sequence of generators $\rho$ is finite. In this subsection, we recall results from \cite{Sjo} and \cite{Hoff3}, which apply in our case.
Assume that $(K,\rho)$ is an existentially closed $G$-field.
\begin{theorem}[Sj\"{o}gren \cite{Sjo}]\label{sjres}
\begin{enumerate}
\item By \cite[Theorem 3]{Sjo}, the field $K$ is PAC.

\item By \cite[Theorem 6]{Sjo}, we have the following isomorphism:
$$\gal(K)\cong \ker(\widetilde{\widehat{G}}\to \widehat{G}).$$
\end{enumerate}
\end{theorem}
If we assume that the theory $\gtcf$ exists, then the appropriate results from  \cite{Hoff3} give:
\begin{itemize}
\item a description of the algebraic closure in models of $\gtcf$;

\item an ``almost quantifier elimination'' for $\gtcf$ (similarly, as in the case of ACFA);

\item a description of the completions of $\gtcf$;

\item the geometric elimination of imaginaries for $\gtcf$.
\end{itemize}
We separately quote one more result from \cite{Hoff3} which is of particular importance for us.
\begin{theorem}[Corollary 4.29 in \cite{Hoff3}]\label{mainhoff}
If $(K,\rho)\models \gtcf$, then $\theo(K)$ is simple if and if the theory $\gtcf$ is simple.
\end{theorem}

\subsection{Simplicity and beyond}

The following lemma will be crucial.
\begin{lemma}\label{crucial}
Suppose $B$ is a profinite group and $A\leqslant_cB$. If the profinite group $K_B:=\ker(\beta:\widetilde{B}\to B)$ is small, then $K_A:=\ker(\alpha:\widetilde{A}\to A)$ is small as well.
\end{lemma}
\begin{proof} By \cite[Remark 16.10.3(d)]{FrJa}, it is enough to show that there is a continuous epimorphism $K_B\to K_A$. Consider the following commutative diagram:
\begin{equation*}
 \xymatrix{ K_A \ar[d]^{}  & & \ker(\beta') \ar[d]^{} \ar[rr]^{=} \ar[ll]_{\pi'} & &  K_B \ar[d]^{} \\
 \widetilde{A} \ar[d]^{\alpha}  & & \widetilde{B}_A  \ar[rr]^{} \ar[ll]_{\pi} \ar[lld]_{\beta'} & &  \widetilde{B} \ar[d]^{\beta} \\
 A \ar[rrrr]^{} & & & &   B,}
\end{equation*}
where $\widetilde{B}_A:=\beta^{-1}(A)$, the epimorphism $\pi$ is given by the universal property of $\widetilde{A}\to A$ (since the profinite group $\widetilde{B}_A$ is projective), $\beta'$ is a restriction of $\beta$, $\pi'$ is a restriction of $\pi$, and the unmarked arrows are inclusions. Since the map $\pi$ is onto, the map $\pi'$ is the continuous epimorphisms which we wanted to show.
\end{proof}
We recall that the \emph{rank} of a profinite group $H$ is the minimal number of its topological generators.
\begin{lemma}\label{pufc}
Suppose $H$ is a pro-$p$-group of rank $m$ and let $\pi:\widehat{F_{m}}(p)\to H$ be a continuous epimorphism. Then $\pi$ is the universal Frattini cover.
\end{lemma}
\begin{proof}
By \cite[Corollary 22.7.8]{FrJa}, the universal Frattini cover of $H$ is of the form $\alpha:\widehat{F_{m}}(p)\to H$. Hence we get an epimorphism $\gamma:\widehat{F_{m}}(p)\to \widehat{F_{m}}(p)$ such that $\alpha\circ\gamma=\pi$. Since $\widehat{F_{m}}(p)$ is Hopfian (see \cite[Prop. 2.5.2]{progps}), the map $\gamma$ is an isomorphism, so $\pi$ is the universal Frattini cover as well.
\end{proof}
The next result must be a folklore, but we could not find a direct reference for it, so we provide a proof.
\begin{lemma}\label{profree}
There is a closed subgroup of $\widehat{F_n}$ which is isomorphic to $\widehat{F_n}(p)$.
\end{lemma}
\begin{proof}
Since $\widehat{F_n}(p)$ is the largest quotient of $\widehat{F_n}$ which is a pro-$p$-group (see e.g. Example 3 on page 7 of \cite{serre2002galois}), there is a continuous epimorphism $\pi:\widehat{F_n}\to \widehat{F_n}(p)$. By \cite[Prop. 22.7.6]{FrJa}, the profinite group $\widehat{F_n}(p)$ is projective, hence $\pi$ has a continuous section $s:\widehat{F_n}(p)\to \widehat{F_n}$ which is a group homomorphism. Since $s$ is continuous and $\widehat{F_n}(p)$ is compact Hausdorff, then $s(\widehat{F_n}(p))$ is a closed subgroup of $\widehat{F_n}$ which is isomorphic to $\widehat{F_n}(p)$.
\end{proof}
We prove below our main result about the kernels of Frattini covers. The proof uses some ideas from Section 8 of \cite{Sjo}.
\begin{theorem}\label{notsmall}
Suppose $G$ is an infinite, finitely generated and virtually free group, which is not free. Then the profinite group
$$\ker(\widetilde{\widehat{G}}\to \widehat{G})$$
is not small.
\end{theorem}
\begin{proof}
Let us fix a free normal subgroup $F_n\triangleleft G$ of finite index. By \cite{Stall1}, $G$ is not torsion-free. Hence, there is a prime $p$ such that $C_p<G$. Clearly, the intersection $F_n\cap C_p$ is trivial, hence $F_n\rtimes C_p$ is a finite index subgroup of $G$. Then $\widehat{F_n\rtimes C_p}$ is a finite index closed subgroup of $\widehat{G}$, so by Lemma \ref{crucial}, we may assume that $G=F_n\rtimes C_p$. Let $\{x_1,\ldots,x_n\}$ be a set of free generators of $F_n$, and $H$ be a subgroup of $G$ generated by the orbit $C_p\cdot x_1$. Since
$$\widehat{H}\rtimes C_p\leqslant_c \widehat{G},$$
we can assume (using Lemma \ref{crucial}) that
$$G=H\rtimes C_p\cong F_k\rtimes C_p,$$
for some $k>0$. Then $G$ is generated by two elements ($x_1$ and a generator of $C_p$), so we have the following exact sequence:
$$1\to F_{\omega}\to F_{2}\to F_k\rtimes C_p\to 1.$$
The kernel is of the right form, by the well-known result saying that if $F$ is a finitely generated free group and $N$ is a nontrivial normal subgroup of infinite index, then $N$ is not finitely generated (see e.g. Exercise 7 in \cite[Section 1.B]{Hatcher}).

The map $F_{2}\to F_k\rtimes C_p$ splits over $F_k$, let $\beta$ be the splitting map and we set:
$$F:=\beta(F_k)<F_2.$$
For any normal subgroup $N<F_{\omega}$ such that $[F_{\omega}:N]$ is a finite power of $p$, we have that $[F_2:NF]$ is a finite power of $p$ and $NF\cap F_{\omega}=F$. Hence, the pro-$p$ topology of $F_2$ induces on $F_{\omega}$ its full pro-$p$ topology, and by \cite[Prop. 3.2.5, Lemma 3.2.6]{progps}, the $p$-profinite completion is an exact functor in this case. Therefore, we get the following exact sequence:
$$1\to \widehat{F_{\omega}}(p)\to \widehat{F_{2}}(p)\to \widehat{F_k}(p)\rtimes C_p\to 1.$$
By Lemma \ref{profree}, the profinite group $\widehat{F_k}(p)\rtimes C_p$ can be considered as a closed subgroup of the profinite group
$$\widehat{F_k\rtimes C_p}=\widehat{F_k}\rtimes C_p.$$
Hence, using Lemma \ref{crucial}, we just need to notice that the map
$$\widehat{F_{2}}(p)\to \widehat{F_k}(p)\rtimes C_p$$
is the universal Frattini cover (since the kernel of this map is the profinite group $\widehat{F_{\omega}}(p)$ which is clearly not small), and this follows from Lemma \ref{pufc}.
\end{proof}
\begin{theorem}\label{nonntp}
Assume that $G$ is a finitely generated, virtually free group. Then the theory
$\gtcf$ is simple if and only if, $G$ is free or $G$ is finite. Moreover, if $\gtcf$ is not simple, then it is not $\ntp_2$. \end{theorem}
\begin{proof}
We already know that if $G$ is free or finite, then the theory $\gtcf$ is simple. Assume that $G$ is infinite, finitely generated, virtually free and not free. Let $(K,\rho)\models G-\mathrm{TCF}$. By Theorem \ref{sjres}(1), the field $K$ is PAC. By Theorem \ref{sjres}(2) and Theorem \ref{notsmall}, the field $K$ is not bounded. It is enough now to use a result of Chatzidakis (see Section 3.5 in \cite{zoepac0}) saying that if a PAC field $K$ is not bounded, then the theory $\theo(K)$ is not NTP$_2$.
\end{proof}
If the marked group $(G,\rho)$ is finitely generated, $(K,\rho)$ is an existentially closed $G$-field and $C$ is its field of constants, then (using results from \cite{Sjo}) $C$ is PAC and the profinite group $\gal(C)$ is finitely generated (being the Frattini cover of the profinite completion of $G$). Therefore, $\gal(C)$ is small and $\mathrm{Th}(C)$ is simple. If we combine this observation with Remark \ref{nonntp}, the conclusion goes quite against our intuition from the ACFA case, where $\theo(K)$ was stable (being algebraically closed) and $C$ was ``the only source of instability''. In our case, after replacing stability with simplicity, the opposite happens: $\theo(K)$ is not simple and $\theo(C)$ is simple.

As it was recently communicated to us by Nick Ramsey, it is likely that for a finitely generated and virtually free group $G$ the theory
$\gtcf$ is NSOP$_1$. If it is the case, then we can use the results of Chatzidakis from \cite{zoepac} about the relations between $\theo(K)$ and $\theo(S\gal(K))$, where for a profinite group $H$, $SH$ is a certain $\omega$-sorted structure which is functorially obtained from $H$.
\begin{remark}
The formalism of $G$-fields includes some cases of \emph{pairs} or \emph{triples} (etc.) of fields, which we explain here. Let
$$G=C_2\ast C_2=\langle \sigma\rangle\ast\langle \tau\rangle\cong D_{\infty}.$$
The structure $(K,\sigma,\tau)$ is inter-definable with the structure $(K,C_{\sigma},C_{\tau})$, where $C_{\sigma},C_{\tau}$ are the corresponding constant fields. Hence this structure can be understood either as a \emph{triple} of fields $(K,C_{\sigma},C_{\tau})$ or perhaps as an \emph{amalgamation} of the fields $C_{\sigma}$ and $C_{\tau}$ (inside the field $K$ which is both definable in the field $C_{\sigma}$ and in the field $C_{\tau}$).

This observation can be generalized to groups of the form $B_1\ast \dots \ast B_k$ for finite $B_i$, and even (with a more complicated amalgamation notion) to the groups considered in Section \ref{secampr}.
\end{remark}

\section{Going further}\label{seclast}
In this section, we discuss possible generalizations of our main result (Theorem \ref{mainthm}). Such generalizations can go into two directions:
\begin{enumerate}
\item finding \emph{necessary} conditions about $G$ for the existence of $\gtcf$;

\item dropping the assumption that $G$ is finitely generated.
\end{enumerate}
In Section \ref{seczrtz}, we give a new example of a group $G$ such that $\gtcf$ does not exist. In Section \ref{secnes}, we state and discuss our conjecture about the item $(1)$ and in Section \ref{secarb}, we discuss the case of arbitrary (i.e. not necessarily finitely generated) groups.

\subsection{The case of $G=\Zz\rtimes \Zz$}\label{seczrtz}
It is known (see \cite{Kikyo1}) that the theory $(\Zz\times \Zz)-\tcf$ does not exist. D. Hoffmann and the second author asked in \cite[Question 5.4(3)]{HK3}
whether it is true that $G-\tcf$ exists if and only if $\Zz\times \Zz$ does not embed into $G$. Let us consider the group $G=\Zz\rtimes \Zz$, where even numbers act trivially on $\Zz$ and odd numbers act by the multiplication by $-1$. We have the following presentation:
$$\Zz\rtimes \Zz=\langle \sigma,\tau\ |\ \tau\sigma=\sigma^{-1}\tau\rangle.$$
In this subsection, we will prove the following.
\begin{theorem}\label{thmzrtz}
For any action of the group $\Zz$ on itself by automorphisms, the theory $(\Zz\rtimes \Zz)-\tcf$ does not exist.
\end{theorem}
Since there are only two actions of $\Zz$ on $\Zz$ by automorphisms (the trivial one and the one described above), in the course of proving Theorem \ref{thmzrtz} we may assume that $G=\Zz\rtimes \Zz$, where the action is described above. We follow (to some extend) Hrushovski's proof of the non-existence of $(\Zz\times \Zz)-\tcf$ as presented in \cite{Kikyo1}. There are several twists in the argument regarding the action of $\tau$ comparing with the case of $G=\Zz\times \Zz$, we comment on them in Remark \ref{twist}.

The equality $\tau\sigma=\sigma^{-1}\tau$ implies the following commutation rules:
$$\sigma^n\tau=\tau\sigma^{-n},\ \ \ \sigma^{-n}\tau=\tau\sigma^n,\ \ \ \sigma\tau^{2n}=\tau^{2n}\sigma,\ \ \ \sigma\tau^{2n+1}=\tau^{2n+1}\sigma^{-1}.$$
Let $\zeta$ be a primitive third root of unity, which is fixed in this subsection.

We extract below the very conclusion of the argument from \cite{Kikyo1}. This conclusion works both for the case of $G=\Zz\times \Zz$ and for the case of $G=\Zz\rtimes \Zz$.
\begin{lemma}\label{a1}
There is no $(\Zz\rtimes \Zz)$-field $(K,\sigma,\tau)$ containing $\zeta$ such that for some $b\in K$ and for some odd $n\in \Nn$ we have:
\begin{enumerate}
\item $\sigma(\zeta)=\tau(\zeta)=\zeta^2$;

\item $\sigma^n(b)=\zeta^ib$ for some $i\in \{0,1,2\}$;

\item $\tau(b)=\zeta\sigma(b)$.
\end{enumerate}
\end{lemma}
\begin{proof}
As at the end of the proof of \cite[Theorem 3.2]{Kikyo1}, we have (using the items $(2)$ and $(3)$):
\begin{IEEEeqnarray*}{rCl}
\sigma^n(\tau(b)) & = & \sigma^n(\zeta\sigma(b))\\
 &= & \sigma^n(\zeta)\sigma(\sigma^{n}(b))\\
 &= & \sigma^n(\zeta)\sigma(\zeta^ib)\\
 & = & \sigma^n(\zeta)\sigma(\zeta^i)\sigma(b).
\end{IEEEeqnarray*}
Using the items  $(1)$ and $(2)$, we get:
\begin{IEEEeqnarray*}{rCl}
b & = & \sigma^{-n}(\zeta^i)\sigma^{-n}(b)\\
 &= & \sigma^{-n}(\zeta)^i\sigma^{-n}(b)\\
 &= & \left(\zeta^2\right)^i\sigma^{-n}(b)\\
 &= & \left(\zeta^i\right)^{-1}\sigma^{-n}(b).
\end{IEEEeqnarray*}
Therefore, we have $\sigma^{-n}(b)=\zeta^ib$, and using the item $(3)$ we get:
\begin{IEEEeqnarray*}{rCl}
\sigma^n(\tau(b)) & = & \tau(\sigma^{-n}(b))\\
 &= & \tau(\zeta^ib)\\
 & = &\sigma(\zeta^i)\zeta\sigma(b).
\end{IEEEeqnarray*}
We obtain that $\sigma^n(\zeta)=\zeta$, which gives a contradiction (as in the proof of \cite[Theorem 3.2]{Kikyo1}), since $\sigma^j(\zeta)=\zeta$ if and only if $j$ is even, and $n$ is odd.
\end{proof}
We prove now the counterpart of \cite[Lemma 3.1]{Kikyo1} for the semi-direct product case.
\begin{lemma}\label{a3}
Assume that $(F,\sigma,\tau)$ is an existentially closed $(\Zz\rtimes \Zz)$-field. Then for any $n>2$, there is $c\in F$ such that
$$\tau(c)=-c,\ \ \  c+\sigma(c)+\ldots+\sigma^{n-1}(c)=0,$$
and for all $k<n$ we have:
$$ c+\sigma(c)+\ldots+\sigma^{k-1}(c)\neq 0.$$
\end{lemma}
\begin{proof}
Let $x_0,\ldots,x_{n-2}$ be algebraically independent over $F$, and we set
$$x_{n-1}:=-(x_0+\dots+x_{n-2}).$$
Then we have:
\begin{itemize}
\item $x_0,\ldots,x_{n-2}$ are algebraically independent over $F$;

\item $x_1,\ldots,x_{n-1}$ are algebraically independent over $F$;

\item $F(x_0,\ldots,x_{n-2})=F(x_1,\ldots,x_{n-1})$.
\end{itemize}
Hence, we can expand $\sigma$ to an automorphism $\sigma'$ of $L:=F(x_0,\ldots,x_{n-2})$ so that
$$\sigma'(x_i)=x_{i+1}$$
for $i\in \{0,\ldots,n-2\}$. In particular we get (it will be useful later):
\begin{IEEEeqnarray*}{rCl}
\sigma'^{-1}\left(x_{n-1}\right) & = & -\left(\sigma'^{-1}(x_0)+\sigma'^{-1}(x_1)+\dots+\sigma'^{-1}(x_{n-2})\right), \\
x_{n-2} & = & -\left(\sigma'^{-1}(x_0)+x_0+\dots+x_{n-3}\right), \\
\sigma'^{-1}(x_0) & = & x_{n-1}, \\
\sigma'(x_{n-1}) &= & x_0.
\end{IEEEeqnarray*}
We also have:
\begin{itemize}
\item $-x_0,-x_{n-1},-x_{n-2},\ldots,-x_2$ are algebraically independent over $F$;

\item $F(x_0,\ldots,x_{n-2})=F(-x_0,-x_{n-1},-x_{n-2},\ldots,-x_2)$.
\end{itemize}
Hence, we can extend $\tau$ to an automorphism $\tau'$ of $F$ in such a way that:
$$\tau'(x_0)=-x_0,\ \ \tau'(x_1)=-x_{n-1},\ \ \tau'(x_2)=-x_{n-2},\ \ \ldots\ ,\ \ \tau'(x_{n-2})=-x_2.$$
We will see that $(L,\sigma',\tau')$ is a $(\Zz\rtimes \Zz)$-extension of $(F,\sigma,\tau)$. Since $(L,\sigma',\tau')$ satisfies our conclusion for $c:=x_0$ and $(F,\sigma,\tau)$ is existentially closed, we will be then done. We check below (on the generators $x_0,\ldots,x_{n-2}$) that $\sigma'\tau'=\tau'\sigma'^{-1}$.
\\
For $i=0$, we compute:
\begin{IEEEeqnarray*}{rCl}
 \tau'\left(\sigma'^{-1}(x_0)\right) & = & \tau'(x_{n-1}) \\
& = & -\left(\tau'(x_0)+\tau'(x_1)+\tau'(x_2)+\ldots+\tau'(x_{n-2})\right) \\
& = & -\left(-x_0-x_{n-1}-x_{n-2}-\ldots-x_{2}\right) \\
& = & x_{n-1}+x_0+x_2+x_3+\ldots+x_{n-2} \\
 &= & -x_1\\
 &= & \sigma'(-x_0)\\
 &= & \sigma'(\tau'(x_0)).
\end{IEEEeqnarray*}
For $i=1$, we compute:
\begin{IEEEeqnarray*}{rCl}
\sigma'\left(\tau'(x_1)\right) & = & -\sigma'\left(x_{n-1}\right) \\
 &= & -x_0\\
 & = & \tau'(x_0) \\
 &= & \tau'\left(\sigma'^{-1}(x_1)\right).
\end{IEEEeqnarray*}
And finally, for $i=2,\ldots,n-2$, we compute:
\begin{IEEEeqnarray*}{rCl}
\sigma'(\tau'(x_i)) & = & \sigma'(-x_{n-i}) \\
 &= & -x_{n-i+1}\\
 & = & \tau'(x_{i-1}) \\
 &= & \tau'\left(\sigma'^{-1}(x_i)\right). \qedhere
\end{IEEEeqnarray*}
\end{proof}
Below is the counterpart of \cite[Claim 3.2.1]{Kikyo1}.
\begin{lemma}\label{a4}
Let $(K,\sigma,\tau)$ be an existentially closed $(\Zz\rtimes \Zz)$-field such that $\zeta\in K$ and $\sigma(\zeta)=\tau(\zeta)=\zeta^2$. Then for any $c\in K$ the following holds:
\\
IF for all $k>0$ we have
$$\tau(c)=-c,\ \ \ c+\sigma(c)+\ldots+\sigma^k(c)\neq 0,$$
THEN there are $a,b\in K$ such that
$$\sigma(a)=a+c,\ \ \ b^3=a,\ \ \ \tau(b)=\zeta\sigma(b).$$
\end{lemma}
\begin{proof}
Let $x$ be a transcendental element over $K$. We extend $\sigma,\tau$ to automorphisms $\sigma',\tau'$ of $K(x)$ such that
$$\sigma'(x)=\tau'(x)=x+c.$$
Then we have $\sigma'^{-1}(x)=x-\sigma^{-1}(c)$, and we compute:
\begin{IEEEeqnarray*}{rCl}
\sigma'\left(\tau'(x)\right) & = & \sigma'\left(x+c\right) \\
 &= & x+c+\sigma(c)\\
 & = & x+c-\sigma(-c) \\
 & = & x+c-\sigma\left(\tau(c)\right) \\
 & = & \tau'(x)-\tau\left(\sigma^{-1}(c)\right) \\
 & = & \tau'\left(x-\sigma^{-1}(c)\right) \\
 &= & \tau'\left(\sigma'^{-1}(x)\right).
\end{IEEEeqnarray*}
Therefore $(K(x),\sigma',\tau')$ is a $(\Zz\rtimes \Zz)$-field.
\\
After setting $c_0:=0$ and $x_0:=x$, for any $n>0$ we have:
$$x_n:=\sigma'^n(x)=x+c+\sigma'(c)+\ldots+\sigma'^{(n-1)}(c)=:x+c_n,$$
$$x_{-n}:=\sigma'^{-n}(x)=x-\sigma'^{-1}(c)-\ldots-\sigma'^{-n}(c)=:x+c_{-n}.$$
It is easy to see (using the assumption about $c$) that for any $i,j\in \Zz$, if $i\neq j$, then we have $c_i\neq c_j$. We also have:
$$\tau'(x_0)=x_1,\ \ \tau'\left(x_{-1}\right)=x_2,\ \ \tau'(x_{-2})=x_3,\ \ \tau'(x_{-3})=x_4,\ \ \ldots\ ;$$
$$\tau'(x_1)=x_0,\ \ \tau'(x_2)=x_{-1},\ \ \tau'(x_3)=x_{-2},\ \ \tau'(x_4)=x_{-3},\ \ \ldots\ .$$
For each $n\in \Zz$, we choose $y_n\in K(x)^{\alg}$ such that $y_n^3=x_n$. Since $(c_i)_{i\in \Zz}$ are pairwise different, using basic Galois theory we obtain that for each $j\in \Zz$, we have
$$y_j\notin K(x)\left(y_i\ |\ i\in\Zz\setminus \{j\}\right).$$
Let $L:=K(x)(y_i\ |\ i\in\Zz)$. Using basic Galois theory again, we see that we have enough freedom to extend $\sigma',\tau'$ to automorphisms of $L$ in the way which is explained below. We extend $\sigma'$ to $L$ by setting $\sigma'(y_i):=y_{i+1}$ for each $i\in \Zz$, and we extend $\tau'$ in the following way:
$$\tau'(y_0)=\zeta y_1,\ \ \tau'(y_{-1})=\zeta^2 y_2,\ \ \tau'(y_{-2})=\zeta y_3,\ \ \tau'(y_{-3})=\zeta^2 y_4,\ \ \ldots\ ;$$
$$\tau'(y_1)=\zeta^2y_0,\ \  \tau'(y_2)=\zeta y_{-1},\ \ \tau'(y_3)=\zeta^2 y_{-2},\ \ \tau'(y_4)=\zeta y_{-3},\ \ \ldots\ .$$
Then the difference field $(L,\sigma',\tau')$ satisfies our conclusion (for $a:=x,b:=y_0$), if we check that $\sigma'\tau'=\tau'\sigma'^{-1}$ on $L$. It is enough to check it on the elements $y_i$ for $i\in \Zz$. We do it below for $i=0,1,-1,2$:
$$\sigma'\left(\tau'(y_0)\right)=\sigma'(\zeta y_1)=\zeta^2 y_2=\tau'(y_{-1})=\tau'\left(\sigma'^{-1}(y_0)\right),$$
$$\sigma'\left(\tau'(y_1)\right)=\sigma'(\zeta^2 y_0)=\zeta y_1=\tau'(y_{0})=\tau'\left(\sigma'^{-1}(y_1)\right),$$
$$\sigma'\left(\tau'(y_{-1})\right)=\sigma'(\zeta^2 y_2)=\zeta y_3=\tau'(y_{-2})=\tau'\left(\sigma'^{-1}(y_{-1})\right),$$
$$\sigma'\left(\tau'(y_2)\right)=\sigma'(\zeta y_{-1})=\zeta^2 y_0=\tau'(y_{1})=\tau'\left(\sigma'^{-1}(y_2)\right).$$
Since $(K,\sigma,\tau)$ is an existentially closed $(\Zz\rtimes \Zz)$-field, it satisfies our conclusion as well.
\end{proof}
The main conclusion is stated (in a rather compact way) below.
\begin{theorem}\label{a5}
There are no $\aleph_0$-saturated, existentially closed $(\Zz\rtimes \Zz)$-fields containing $\zeta$ and such that $\sigma(\zeta)=\tau(\zeta)=\zeta^2$.
\end{theorem}
\begin{proof}
Suppose not, and let $(K,\sigma,\tau)$ be an existentially closed $(\Zz\rtimes \Zz)$-field, which is $\aleph_0$-saturated and such that $\sigma(\zeta)=\tau(\zeta)=\zeta^2$. By saturation and Lemma \ref{a4}, there is $n_0>0$ such that for any $c\in K$, if for all $k<n_0$ we have
$$\tau(c)=-c,\ \ \ c+\sigma(c)+\ldots+\sigma^k(c)\neq 0,$$
then there are $a,b\in K$ such that
$$\sigma(a)=a+c,\ \ \ b^3=a,\ \ \ \tau(b)=\zeta\sigma(b).$$
By Lemma \ref{a3}, there is $c\in K$ such that for some odd $n>n_0$ we have
\begin{equation}
c+\sigma(c)+\ldots+\sigma^{n-1}(c)=0,\tag{$*$}
\end{equation}
and for all $k<n$ (so also for all $k<n_0$), we have:
$$\tau(c)=-c,\ \ \ c+\sigma(c)+\ldots+\sigma^k(c)\neq 0.$$
Hence, there are $a,b\in K$ such that
\begin{equation}
\sigma(a)=a+c,\ \ \ b^3=a,\ \ \ \tau(b)=\zeta\sigma(b).\tag{$**$}
\end{equation}
Using $(**)$ and $(*)$, we get:
$$\sigma^n(a)=a+c+\sigma(c)+\dots+\sigma^{n-1}(c)=a.$$
Since $b^3=a$, we get $(\sigma^n(b))^3=\sigma^n(a)=a$, so $\sigma^n(b)$ is a third root of $a$. Since $b$ is a third root of $a$ as well, there is $i\in \{0,1,2\}$ such that $\sigma^n(b)=\zeta^ib$.

Hence we are in the situation from Lemma \ref{a1}, so we get a contradiction.
\end{proof}
\begin{remark}\label{twist}
As we have pointed out, the proof of Theorem \ref{a5} follows the lines of the proof from \cite{Kikyo1}. The only technical difference is that the condition ``$\sigma(c)=\tau(c)$'' from \cite{Kikyo1} is replaced here by the condition ``$\tau(c)=-c$''. This change made some computations slightly more involved.
\end{remark}
\begin{cor}
The theory $(\Zz\rtimes \Zz)-\tcf$ does not exist.
\end{cor}
\begin{proof}
Let $F=\Qq(\zeta)$ and $\sigma,\tau\in \aut(F)$ be such that $\sigma(\zeta)=\beta(\zeta)=\zeta^2$. Since $\sigma=\sigma^{-1}$ and $\sigma\tau=\tau\sigma$, the difference field $(F,\sigma,\tau)$ is a $(\Zz\rtimes \Zz)$-field. Then $(F,\sigma,\tau)$ has an existentially closed $(\Zz\rtimes \Zz)$-extension $(K,\sigma',\tau')$. If the theory $(\Zz\rtimes \Zz)-\tcf$ exists, then we can take $(K,\sigma',\tau')$ to be $\aleph_0$-saturated, which contradicts Theorem \ref{a5}.
\end{proof}

\begin{remark}\label{dirsum}
We can easily extend the results of this subsection to the case of groups $G$ which are of the following form:
$$G \cong (\Zz\rtimes \Zz)\rtimes H.$$
We just need to notice that we can always expand the actions of $\Zz\rtimes \Zz$ on the rings of polynomials to $G$, by setting $h(x_i)=x_i$ for each $h\in H$. More precisely, first  we put a $G$-ring structure on $\Qq[X]$ by applying the splitting epimorphism $G\to \Zz\rtimes \Zz$ and whichever construction for $\Zz\rtimes \Zz$ we performed above. Then, using the following ring isomorphism:
$$F[X]\cong \Qq[X]\otimes_{\Qq}F,$$
we put on $F[X]$ the tensor product $G$-ring structure and extend it to $F(X)$.
\end{remark}

\subsection{Necessity}\label{secnes}
As mentioned in Introduction, virtually free groups have many equivalent characterizations coming from different branches of mathematics (see e.g. Introduction in \cite{arauja}). We conjecture below that there is one more coming from model theory.
\begin{conjecture}\label{conj1}
Suppose that $G$ is a finitely generated group. Then the theory $\gtcf$ exists if and only if $G$ is virtually free.
\end{conjecture}
It is not clear to us how to attack Conjecture \ref{conj1}. To see the possible problems, we point out that one would have to deal with complicated groups as
infinite \emph{Burnside groups}, which are finitely generated and not virtually free (they are even torsion of bounded exponent).

On the positive side, we can confirm Conjecture \ref{conj1} for commutative groups.
\begin{theorem}
Suppose that $G$ is a finitely generated commutative group. Then the theory $\gtcf$ exists if and only if the group $G$ is virtually free.
\end{theorem}
\begin{proof}
By Theorem \ref{mainthm}, we get the right-to-left implication. For the other implication, assume that $G$ is not virtually free. By the structure theorem for finitely generated commutative groups, the group $\Zz\times \Zz$ is a direct summand of $G$. By Remark \ref{dirsum}, the theory $\gtcf$ does not exist.
\end{proof}
In the course of proving Theorem \ref{mainthm}, we have shown that if $G$ is finitely generated and virtually free, then any (finitely generated) $G$-kernel has a prolongation which is a $G$-extension. It would be interesting to know whether this is actually an algebraic description of the existence of the model companion, so we ask the following.
\begin{question}\label{kerprolmc}
Does the theory $G-\tcf$ exist if and only if any (finitely generated) $G$-kernel has a prolongation which is a $G$-extension?
\end{question}
It is possible that the arguments from Section \ref{seczrtz} can be used to show that some $(\Zz\rtimes \Zz)$-kernels do not have prolongations which are ($\Zz\rtimes \Zz$)-field extensions.

\subsection{Arbitrary groups}\label{secarb}
To consider the case of arbitrary groups, one needs to put the (additive) group $\Qq$ (since $\Qq-\tcf=\Qq$ACFA exists, see \cite{med1}) and virtually free groups into a common context. However, it is easy to see that each virtually free group is also \emph{locally virtually free} (i.e. each finitely generated subgroup is virtually free), and $\Qq$ is locally virtually free (even locally cyclic, which was crucial in \cite{med1}) as well. Hence we can generalize Conjecture \ref{conj1} the following.
\begin{conjecture}\label{conj2}
Let $G$ be a group. Then the theory $\gtcf$ exists if and only if $G$ is locally virtually free.
\end{conjecture}
The first, and hopefully relatively easy, step towards proving Conjecture \ref{conj2} would be trying to extend the methods from \cite{med1} to the case of an arbitrary locally virtually free group. Namely, we have:
$$\Qq\cong \coli_{n}\frac{1}{n}\Zz,$$
and the axioms of $\Qq-\tcf$ from \cite{med1} can be understood as some kind of ``limit'' of the axioms of $(\frac{1}{n}\Zz)-\tcf$(=ACFA). Similarly, for any locally virtually free group $G$, we have
$$G\cong \coli_{n}G_n,$$
where each $G_n$ is virtually free, and one can hope to obtain the axioms of $\gtcf$ as some kind of limit of the axioms of $G_n-\tcf$ (given by Theorem \ref{mainthm}).

There is an interesting parallel between the paragraph above and the axiomatization of existentially closed fields with iterative Hasse-Schmidt derivations (in the case of positive characteristic). For the necessary background, the reader is advised to consult e.g. \cite{HK}. Let $\ga$ be the additive group over a field of positive characteristic $p$, $\ga[n]$ be the kernel (considered as a finite group scheme) of the algebraic group morphism $\fr^n_{\ga}$ and $\gaf$ be the formal group scheme which is the formalization of $\ga$. Then we have:
$$\gaf\cong \coli_{n}\ga[n],$$
the algebraic actions of $\gaf$ correspond to iterative Hasse-Schmidt derivations and the algebraic actions of $\ga[n]$ correspond to $n$-truncated iterative Hasse-Schmidt derivations. Then indeed, the theory of existentially closed fields with iterative Hasse-Schmidt derivations can be considered as the limit of the theories of existentially closed fields with $n$-truncated iterative Hasse-Schmidt derivations (see \cite{K3} and, in a more general context, \cite{HK}).

The common context between $G$-fields and fields with iterative Hasse-Schmidt derivations (i.e. also differential fields) would be fields with algebraic actions of a fixed formal group scheme. A general theory of such actions is considered in e.g. \cite{MS}, \cite{MS2} and \cite{Kam}. However, we would rather not state any conjecture for the existence of a model companion in such a general case.

\bibliographystyle{plain}
\bibliography{harvard}

\end{document}